\documentclass[11pt,reqno]{amsart}
\usepackage[foot]{amsaddr}

\usepackage{amsmath} 
\usepackage{amsfonts,amssymb,amsthm,bm}
\usepackage{amsrefs}
\usepackage{enumitem}
\usepackage[utf8]{inputenc}
\usepackage[T1]{fontenc}
\usepackage[dvipsnames]{xcolor}
\usepackage{stackrel}

\usepackage{fullpage}

\def\dd{\mathrm{\,d}} 
\newcommand\abs[1]{\left|#1\right|} 
\newcommand\norm[1]{\left\lVert#1\right\rVert} 
\newcommand*\lap{\mathop{}\!\mathbin\Delta} 
\newcommand*\grad{\mathop{}\!\mathbin\nabla} 
\renewcommand*\div{\mathop{}\!\mathbin\nabla\cdot} 

\DeclareMathOperator{\diam}{diam} 
\newcommand{\Chi}{\mathcal{X}} 
\let\emptyset\varnothing 
\DeclareMathOperator{\Lip}{Lip} 
\DeclareMathOperator{\ch}{ch} 
\DeclareMathOperator*{\essinf}{ess\,inf}
\DeclareMathOperator*{\esssup}{ess\,sup}
\DeclareMathOperator*{\loc}{loc} 

\def\Xint#1{\mathchoice
{\XXint\displaystyle\textstyle{#1}}%
{\XXint\textstyle\scriptstyle{#1}}%
{\XXint\scriptstyle\scriptscriptstyle{#1}}%
{\XXint\scriptscriptstyle\scriptscriptstyle{#1}}%
\!\int}
\def\XXint#1#2#3{\mkern3mu{\setbox0=\hbox{$#1{#2#3}{\int}$ }
\vcenter{\hbox{$#2#3$ }}\kern-.6\wd0}}

\def\dashint{\Xint-}

\theoremstyle{plain}
\newtheorem{theorem}{Theorem}

\newtheorem{lemma}[theorem]{Lemma}
\numberwithin{theorem}{section}

\theoremstyle{remark}
\newtheorem{remark}[theorem]{Remark}

\theoremstyle{definition}
\newtheorem{dfn}[theorem]{Definition}

\begin{document}

\nocite{*}

\title{Weighted norm inequalities in a bounded domain by the sparse domination method}

\author{Emma-Karoliina Kurki}
\address[E-K.K.]{Aalto University, Department of Mathematics and Systems Analysis, P.O. Box 11100, FI-00076 Aalto, Finland}
\email{emma-karoliina.kurki@aalto.fi}
\author{Antti V. Vähäkangas}
\address[A.V.V.]{University of Jyvaskyla, Department of Mathematics and Statistics, P.O. Box 35, FI-40014 University of Jyvaskyla, Finland} 
\email{antti.vahakangas@iki.fi}
\thanks{The authors would like to thank Lucio Boccardo and Juha Kinnunen for the research question, as well as Andrei Lerner and Ritva Hurri-Syrjänen for pointing out valuable references. Emma-Karoliina Kurki thanks the Emil Aaltonen Foundation for the young researcher's grant that made this project possible. This research was supported by the Academy of Finland.}
\date{\today}

\begin{abstract}
 We prove a local two-weight Poincaré inequality 
for cubes using the sparse domination method that has been influential in harmonic analysis. The proof involves a localized version of the Fefferman--Stein inequality for the sharp maximal function. 
By establishing a local-to-global result in a bounded domain satisfying a Boman chain condition, we show a two-weight $p$-Poincar\'e inequality in such domains. As an application we show that certain nonnegative supersolutions of the $p$-Laplace equation and distance weights are $p$-admissible in a bounded domain, in the sense that they support versions of the $p$-Poincar\'e inequality.
\end{abstract}

\maketitle

\section{Introduction}

Poincaré inequalities are useful tools in analysis, especially so in the theory of partial differential equations (PDEs). The present paper stems from the question whether it is possible to establish a weighted Poincaré inequality with respect to weights defined on a bounded domain only. This setting is natural from the viewpoint of analysis of PDEs and calls for a localized argument, which we borrow from harmonic analysis.

To begin with, we provide sufficient conditions for a general two-weight Sobolev--Poincaré inequality
\begin{equation}\label{e.general_two_w}
\left(\int_{Q_0}\abs{u(x)-u_{Q_0}}^qw(x)\dd x\right)^\frac{1}{q} \leq C\left(\int_{Q_0} \abs{\grad u(x)}^pv(x)\dd x\right)^{\frac{1}{p}}
\end{equation}
to hold for every $u\in\Lip(Q_0)$ with $1<p\leq q<\infty$. This result is Theorem \ref{lokaali} below, which is
applied to show that \eqref{e.p_poincare} holds in cubes $Q_0$ with $4Q_0\subset \Omega$ and $\frac{2n}{n+1}<p<\infty$, including the case $p=2$ of the classical Laplacian. Weighted norm inequalities such as the above are relevant to the study of nonlinear PDEs and so-called \emph{$p$-admissible} weights; for instance, see Chapter 20 in \cite{MR2305115}. Furthermore, the proof quite naturally yields a local variant of the Fefferman–Stein inequality for the so-called \emph{sharp maximal function} (Theorem \ref{maxima}).

Having established the local inequality \eqref{e.general_two_w}, we extend it to domains satisfying a \emph{Boman chain condition}. The geometry of Boman domains allows us to propagate the estimate from single Whitney cubes to the entire domain adapting an idea of Iwaniec and Nolder \cite{ino}; see also Chua \cite{MR1140667}. As a result, in Theorem \ref{main} we present the two-weight inequality
\begin{equation}\label{e.main}
\left(\inf_{c\in\mathbb{R}^n} \int_\Omega \abs{u(x)-c}^qw(x)\dd x\right)^\frac{1}{q} \leq C\left(\int_\Omega \abs{\grad u(x)}^pv(x)\dd x\right)^\frac{1}{p}
\end{equation}
under the assumption that the weights $w$ and $\sigma = v^{-1/(p-1)}$  satisfy suitable local doubling and $A_\infty$ conditions, as well as a Muckenhoupt-type compatibility condition in dilated Whitney cubes and their dyadic subcubes $Q^*\subset\Omega$:
\begin{equation*}
\left(\frac{1}{\abs{Q^*}^{1-1/n}}\right)^pw(Q^*)^{\frac{p}{q}}\sigma(Q^*)^{p-1} \leq K.
\end{equation*}

For the sake of demonstration, we present two applications to distance weights, as well as one for nonnegative supersolutions to the $p$-Laplace equation. Namely, we establish the weighted Poincar\'e inequality 
\begin{equation}\label{e.p_poincare}
\int_{Q_0}\abs{u(x)-u_{Q_0}}^pw(x)\dd x \leq Cl(Q_0)^p\int_{Q_0} \abs{\grad u(x)}^pw(x)\dd x,
\end{equation}
where $u\in\Lip(Q_0)$, and $C = C(n, p, \Omega) > 0$. This inequality is known to hold for all cubes $Q_0\subset\mathbb{R}^n$ whenever $w\in W^{1,p}_{\loc}(\mathbb{R}^n)$ is a nonnegative weak supersolution to the $p$-Laplace equation in $\mathbb{R}^n$ and $\frac{2n}{n+1}<p<\infty$; see \cite{MR2305115}*{p. 74}. We show a corresponding local variant of this result. We assume that the weight $w\in W^{1,p}_{\loc}(\Omega)$ is a nonnegative weak supersolution of the $p$-Laplace equation in a bounded domain $\Omega\subset\mathbb{R}^n$, and prove \eqref{e.p_poincare} in cubes $Q_0\subset\Omega$ such that $4Q_0\subset\Omega$. 

A key new feature of Theorem \ref{lokaali} is that the assumptions concerning the weights $v$ and $w$ are localized to the fixed cube $Q_0$, which is indeed necessary in order to establish the local version of \eqref{e.p_poincare} in Theorem \ref{plapplication_thm}. This distinguishes our main result \eqref{e.main} (Theorem \ref{main}) from other two-weight inequalities in the same vein. In particular, our result provides a more strictly localized variant of a theorem due to Chua \cite{MR1140667}*{Theorem 2.14}. We only assume the weight $w$ to be doubling in $\Omega$, which is a weaker notion than the global doubling required by Chua. The downside of our approach are the ensuing supplementary $A_\infty$ conditions. These can be seen as an artefact of the method, which will be discussed next. Other sufficient conditions for two-weight inequalities are given by Chanillo and Wheeden \cite{MR1189903}, and Turesson \cite{ref_Tur}*{Theorem 2.6.1}. Maz'ya \cite{ref_Maz}*{Section 3.8} presents a necessary and sufficient capacitary condition for fractional Sobolev spaces. See also  Dyda et al.\ \cite{MR3900847}, Hurri and Hurri--Syrj\"anen \cites{Hurri1990,Hurri1992}, and Muckenhoupt--Wheeden's early paper \cite{MR0340523}.

In order to prove \eqref{e.general_two_w}, we proceed in two stages. Both are based on the idea of \emph{sparse domination}, in which one first provides a pointwise inequality in terms of a sparse dyadic operator. Consequently, the problem is reduced to showing a uniform weighted norm inequality for a significantly simpler class of sparse dyadic operators. In the first stage, we follow the idea in \cite{MR3695871} and show the pointwise inequality 
\begin{equation}\label{e.sparse_demo}
\abs{f(x)-f_{Q_0}}\leq C\sum_{Q\in\mathcal{S}}\Chi_Q(x)\dashint_Q\abs{f(y)-f_Q}\dd y,
\end{equation}
where the collection $\mathcal{S}$ of dyadic cubes inside $Q_0$ depends on $f$, and is \emph{sparse} in the sense that there are pairwise disjoint, measurable subsets $E_S\subset S\in\mathcal{S}$ incorporating a sufficiently large part of $S$ in terms of weighted measure. Moreover, the constant $C$ is independent of $f$, and therefore sharp maximal functions can be used to control the sparse dyadic operator on the right-hand side of \eqref{e.sparse_demo}. The Poincar\'e inequality, followed by another pointwise sparse domination argument pertaining to the fractional maximal function and due to Pérez \cite{MR1052009}, allows us to conclude the proof of \eqref{e.general_two_w}.
 
To the authors' knowledge, the application of the sparse domination method to Poincar\'e inequalities is new, and the present work could also be considered a demonstration of its scope of application. The method has been very influential in harmonic analysis, in which several weighted inequalities for singular integrals and potential operators have been established by using the idea of pointwise dyadic domination, including the celebrated $A_2$ theorem. For a selection of recent examples and developments, we refer to Pereyra's lecture notes \cite{revo}. Particularly influential works for our purposes include those of Lerner, such as \cite{MR3695871} (with Ombrosi and Rivera-Ríos) and \cite{MR2721744}. Compared to state-of-the-art instances of the sparse domination argument, our version is vastly simpler, yet perfectly sufficient for its purpose and with the further advantage of being localized. 

\section{Setup}

 A \emph{cube} in $\mathbb{R}^n$ is a half-open set of the form
\[
Q = \left[a_1,b_1\right)\times\ldots\times\left[a_n,b_n\right),
\]
with $b_1-a_1=\ldots=b_n-a_n$. For our purposes, a cube $Q = Q(x_Q,r_Q)$ is most conveniently determined by its midpoint $x_Q$ and side length $2r_Q=b_1-a_1$. We denote the side length of a given cube $Q$ by $l(Q)$.  If $N>0$, we also adopt the shorthand notation $NQ=Q(x_Q,Nr_Q)$.

We denote the characteristic function of a set $E\subset X$ by $\Chi_{E}$, that is, $\Chi_{E}(x)=1$ if $x\in E$ and $\Chi_{E}(x)=0$ if $x\in X\setminus E$.

Any open, nontrivial, proper subset $\Omega\subset\mathbb{R}^n$ admits a \emph{Whitney decomposition}, which we denote by $\mathcal{W}=\mathcal{W}(\Omega)$. The standard construction can be found e.~g. in \cite{MR2445437}, Appendix J. Choosing cubes to be half-open, the Whitney cubes are disjoint. Furthermore, they cover the open set $\Omega$: $\cup_{Q\in\mathcal{W}}Q = \Omega$. For a cube $Q = Q(x,r)\in\mathcal{W}$, the corresponding \emph{dilated cube} is denoted by $Q^* = \frac{9}{8}Q = Q(x,\frac{9}{8}r)$. Such dilated cubes have bounded overlap, which means that $\sum_{Q\in\mathcal{W}}\Chi_{Q^*}\leq C(n)$. 
Moreover, since the side length of a dilated Whitney cube $Q^*$ is comparable to its distance from the boundary of the set, there exists a constant $C=C(n)$ such that 
\begin{equation}\label{sideledil}
\frac{l(Q^*)}{C(n)} \leq d(Q^*,\partial\Omega) \leq C(n)l(Q^*). 
\end{equation}

The family of Lipschitz continuous functions on a set $\Omega$ is denoted $\Lip(\Omega)$. \emph{Local} classes of functions mean that the property in question holds for every compact set $K\subset\Omega$; these are indicated with a subscript, such as $\Lip_{\loc}(\Omega)$ for the set of locally Lipschitz continuous functions on $\Omega$.

For a cube $Q_0\subset\mathbb{R}^n$, the collection of its \emph{dyadic children}, denoted $\ch_\mathcal{D}(Q_0)$, are the $2^n$ cubes with side length $l(Q)/2$ obtained by bisecting each edge. Continuing this process recursively, we obtain the infinite collection $\mathcal{D}(Q_0)$ of \emph{dyadic subcubes} that consists of $Q_0$ and its dyadic descendants in any generation. We will be routinely making use of the fact that these cubes are nested: if $Q, Q'\in\mathcal{D}(Q_0)$, then either one is contained in the other or the cubes are disjoint.  Thanks to the nestedness property, each cube in $\mathcal{D}(Q_0)\setminus \{Q_0\}$ has a unique \emph{dyadic parent}, denoted $\pi Q$: the cube $Q'\in\mathcal{D}(Q_0)$ such that $Q\in\ch_\mathcal{D}(Q')$. 

With respect to a generic collection $\mathcal{E}\subset\mathcal{D}(Q_0)$ of cubes with $Q_0\in\mathcal{E}$, the \emph{$\mathcal{E}$-parent} $\pi_\mathcal{E}Q$ of any dyadic cube $Q\subset Q_0$ is the minimal cube in $\mathcal{E}$ that contains $Q$. Notice that $Q$ is not necessarily in $\mathcal{E}$ to begin with, and the inclusion need not be strict, so $\pi_\mathcal{E}Q=Q$ whenever $Q\in\mathcal{E}$. 
Analogously to the dyadic case, 
The $\mathcal{E}$-children $\ch_\mathcal{E}(Q)$ of a cube $Q\in\mathcal{E}$ are the maximal cubes in $\mathcal{E}$ strictly contained in $Q$.

In the following, subcubes that are constructed beginning with a fixed cube $Q_0$ will be referred to as ``dyadic cubes''. In other words, all ``dyadic cubes'' are $Q_0$-dyadic, whether or not this is spelled out. In the first part, we will be operating inside a cube $Q_0\subset\mathbb{R}^n$; in the proof of the local-to-global result, $Q_0$ will turn out to be a dilatation of a cube $Q\subset\mathcal{W}(\Omega)$ of the Whitney decomposition. 

We say that a locally integrable function $w$ is a \emph{weight} in an open set $\Omega$, if $w(x)>0$ for
almost every $x\in \Omega$. 
The weighted measure (or ``weight'') of a measurable set $E\subset \Omega$,  in our case typically a cube, with respect to $w$ is 
\begin{equation*}
w(E) = \int_E w(x)\dd x.
\end{equation*}
The integral average of a  function $f\in L^1(E)$ over a measurable set $E\subset \Omega$ is written $f_E$ for short, and the corresponding average with respect to a weight $w$ is indicated by adding another subscript:
\begin{equation*}
f_{w;E} = \frac{1}{w(E)}\int_Ef(x)w(x)\dd x.
\end{equation*}

Throughout the proof of the local result, we are dealing with Muckenhoupt $A_\infty$ weights in the cube $Q_0$. In fact, we only need to assume the $A_\infty$ property in dyadic cubes, as will be detailed shortly. 

\begin{dfn}\label{ainfty0} 
A  weight  $w$ in a cube $Q_0$ belongs to the \emph{dyadic Muckenhoupt $A_\infty$} class, denoted $w\in A^d_\infty(Q_0)$, if there exist constants $C_w$ and $\delta_w$ in $\left(0,\infty\right)$ such that for all dyadic cubes $Q\subset Q_0$ and all measurable subsets $E\subset Q$ we have
\begin{equation*}
\frac{w(E)}{w(Q)}\leq C_w\left(\frac{\abs{E}}{\abs{Q}}\right)^{\delta_w}. 
\end{equation*}
\end{dfn}

We will need to make intermediate estimates in terms of certain maximal functions, which are introduced next.
\begin{dfn} For a cube $Q_0$, $0\leq\alpha<n$ and $f\in L^1(Q_0)$, we define the \emph{dyadic fractional maximal function}
\begin{equation*}
M^d_{\alpha, Q_0}f(x) = \sup_{\substack{Q\subset Q_0\\ Q\ni\,x}}\frac{1}{\abs{Q}^{1-\alpha/n}}\int_Q\abs{f(y)}\dd y,
\end{equation*}
where the supremum is taken over all dyadic cubes $Q\subset Q_0$ such that $x\in Q$. 
\end{dfn}
\begin{dfn} Let $Q_0$ be a cube, $w$ a weight in $Q_0$, and $f\in L^1(Q_0; w\dd x)$. We define \emph{the weighted dyadic maximal function}
\begin{equation*}
M^{d, w}_{Q_0}f(x) = \sup_{\substack{Q\subset Q_0\\ Q\ni\,x}}\frac{1}{w(Q)}\int_Q\abs{f(y)}w(y)\dd y,
\end{equation*}
again taking the supremum over all dyadic cubes $Q\subset Q_0$ such that $x\in Q$. 
\end{dfn}
The following standard lemma shows that the weighted dyadic maximal function is bounded in $L^p(Q_0; w\dd x)$.
\begin{lemma}\label{strong} Let $Q_0\subset\mathbb{R}^n$ a cube, $1<p<\infty$, and $w$ a weight in $Q_0$. Then there is a constant $C = C(p) > 0$ such that for all $f\in L^p(Q_0; w\dd x)$
\begin{equation*}
\int_{Q_0}\left(M^{d, w}_{Q_0}f(x)\right)^pw(x)\dd x \leq C\int_{Q_0}\abs{f(x)}^pw(x)\dd x.
\end{equation*}
\end{lemma}
\begin{proof}
The statement for $1<p<\infty$ is proven by interpolation. To begin with, we immediately observe that $\lVert M^{d,w}_{Q_0}f\rVert_{L^\infty(Q_0;w\dd x)}\leq \norm{f}_{L^\infty(Q_0;w\dd x)}$. It remains to verify that the maximal operator is of weak type $(1,1)$. To this end, let $f\in L^1(Q_0; w\dd x)$ and $t>0$. We claim that 
\begin{equation}\label{weako}
w\left(\{x\in Q_0\mathbin{:} M^{d,w}_{Q_0}f(x) > t\}\right) \leq\frac{1}{t}\int_{Q_0}\abs{f(x)}w(x)\dd x.
\end{equation}
For brevity, denote $E_t = \{x\in Q_0\mathbin{:} M^{d,w}_{Q_0}f(x) > t\}$. To see \eqref{weako}, fix $t>0$ and consider the collection of all dyadic cubes $Q\subset Q_0$ satisfying 
\begin{equation}\label{condi}
\frac{1}{w(Q)}\int_Q\abs{f(y)}w(y)\dd y > t.
\end{equation}
If the collection is empty, we have $M^{d,w}_{Q_0}f(x)\leq t$ almost everywhere in $Q_0$ and $w(E_t) = 0$. Otherwise, we fix a collection of maximal cubes $\left\{Q_i\right\}$ satisfying \eqref{condi}. In particular, these cubes are pairwise disjoint. We have $E_t = \bigcup_{i=1}^\infty Q_i$, and
\begin{align*}
w(E_t) & \leq w\left(\bigcup_{i=1}^\infty Q_i\right) = \sum_{i=1}^\infty w(Q_i)\\
& \leq \frac{1}{t}\sum_{i=1}^\infty\int_{Q_i}\abs{f(x)}w(x)\dd x = \frac{1}{t}\int_{E_t} \abs{f(x)}w(x)\dd x\\
& \leq \frac{1}{t}\int_{Q_0} \abs{f(x)}w(x)\dd x, 
\end{align*}
which proves \eqref{weako}. Applying the Marcinkiewicz interpolation theorem (\cite{MR2445437}, Theorem 1.3.2), we obtain for $1<p<\infty$
\begin{equation*}
\norm{M^{d,w}_{Q_0}f}_{L^p(Q_0;w\dd x)}^p \leq \frac{p2^p}{p-1}\norm{f}_{L^p(Q_0; w\dd x)}^p. 
\end{equation*}
\end{proof}

\section{Sparse domination I}

The main result in this section is the sparse domination lemma \ref{SparseDom}, a weighted variant of Lemma 5.1 in \cite{MR3695871}. The auxiliary lemma \ref{SparsePre} will provide us with a sparse collection of cubes that appears in the sparse domination lemma. More specifically, the condition \eqref{S3} means that the $\mathcal{S}$-children
of any $S\in\mathcal{S}$ only occupy a controlled fraction of the $w$-measure of $S$. In Lemma \ref{SparseDom}, this property will be used to build a pairwise disjoint collection of sets $E_S\subset S$, each of which has a large $w$-measure compared to that of $S\in\mathcal{S}$, controlled by a parameter $\eta>0$ depending on the $A_\infty$ constants of $w$. This is what is meant by saying that $\mathcal{S}$ is a {\em sparse collection of cubes.} 

 In the sparse domination lemma \ref{SparseDom}, sparse collections provided by the auxiliary lemma are used to define dyadic sparse operators of the form
\[
\sum_{Q\in\mathcal{S}}\Chi_Q(x)\dashint_Q\abs{f(y)-f_Q}\dd y.
\]
Ultimately due to the inequality \eqref{S1}, these dyadic sparse operators can be used to dominate the quantity $\lvert f(x)-f_{Q_0}\rvert$ pointwise. On the other hand, the dyadic sparse operators can be easily controlled by duality and maximal function arguments due to existence of the disjoint sets $E_S$. This also reflects the general principles behind the sparse domination paradigm.

\begin{lemma}\label{SparsePre}
Let $Q_0$ be a cube in $\mathbb{R}^n$, and $f\in L^1(Q_0)$. Let $w\in A^d_\infty(Q_0)$ with $C_w>0$ and $\delta_w>0$.
There exists a collection $\mathcal{S}$ of dyadic cubes such that each cube $S\in\mathcal{S}$ satisfies
\begin{enumerate}[label=\normalfont{(\alph*)}]
\item \label{syksi} If $Q\subset Q_0$ is dyadic cube such that $\pi_\mathcal{S}Q = S$, then
\begin{equation}\label{S1}
\dashint_Q\abs{f(x)-f_S}\dd x\leq\rho\dashint_S\abs{f(x)-f_S}\dd x. 
\end{equation}
\item \label{skolme} There is a constant $\rho = \rho(C_w,\delta_w) >1$ such that 
\begin{equation}\label{S3}
\sum_{S'\in\ch_\mathcal{S}(S)}w(S') \leq C_w\rho^{-\delta_w}w(S) < w(S). 
\end{equation}
\end{enumerate}
\end{lemma}
\begin{proof}
We will construct the collection $\mathcal{S}$ by a stopping-time argument. To begin with, fix a function $f\in L^1(Q_0)$ and a constant $\rho = \rho(C_w,\delta_w) >1$ such that $C_w\rho^{-\delta_w}<1$. We may assume that $\dashint_{Q_0}\abs{f(x)-f_{Q_0}}\dd x>0$; otherwise, $f$ is constant at the Lebesgue points of $Q_0$ and we can take $\mathcal{S}=\{Q_0\}$. First, we place $Q_0$ inside $\mathcal{S}$ and proceed recursively: for each $Q_0$-dyadic cube $S\in\mathcal{S}$, we add to $\mathcal{S}$ the maximal dyadic cubes $S'\subset S$ that satisfy the stopping condition
\begin{equation}\label{stopping}
\dashint_{S'}\abs{f(x)-f_S}\dd x>\rho\dashint_S\abs{f(x)-f_S}\dd x. 
\end{equation}
This process is iterated \emph{ad infinitum} if necessary. As a result, we obtain a collection $\mathcal{S}$ of dyadic cubes in $Q_0$. Claim \ref{syksi} is an immediate consequence of the stopping-time construction. Namely, let $S\in\mathcal{S}$. Recall that if $Q\subset Q_0$ is a cube, $\pi_\mathcal{S}Q$ is the minimal cube in $\mathcal{S}$ that contains $Q$. If $Q\subset Q_0$ is a dyadic cube such that $\pi_\mathcal{S}Q = S$, then
\begin{equation*}
\dashint_Q\abs{f(x)-f_S}\dd x\leq\rho\dashint_S\abs{f(x)-f_S}\dd x.
\end{equation*}

As for \ref{skolme}, fix $S\in\mathcal{S}$ and recall that $\ch_\mathcal{S}(S)$ are the the maximal cubes in $\mathcal{S}$ that are strictly contained in $S$. In particular, we notice that this collection is disjoint. As per the stopping condition \eqref{stopping}, for every $S'\in \ch_\mathcal{S}(S)$ it holds that
\begin{equation*}
\dashint_{S'}\abs{f(x)-f_S}\dd x > \rho\dashint_S\abs{f(x)-f_S}\dd x.
\end{equation*}
Since the collection $\ch_\mathcal{S}(S)$ is disjoint, the $A_\infty^d(Q_0)$ condition of $w$ implies that for all $S\in\mathcal{S}$
\begin{align*}
\frac{\sum_{S'\in\ch_\mathcal{S}(S)}w(S')}{w(S)} & \leq C_w\left(\frac{\sum_{S'\in\ch_\mathcal{S}(S)}\abs{S'}}{\abs{S}}\right)^{\delta_w}\\
& \leq C_w \rho^{-\delta_w}  \left(\sum_{S'\in\ch_\mathcal{S}(S)}\frac{\int_{S'}\abs{f(x)-f_S}\dd x}{\int_S\abs{f(x)-f_S}\dd x}\right)^{\delta_w} \leq C_w\rho^{-\delta_w} < 1. 
\end{align*}
\end{proof}

With our sparse collection of cubes at hand, we are now set to show the existence of the pairwise disjoint sets $E_Q$ and to estimate $\lvert f(x)-f_{Q_0}\rvert$ pointwise with a dyadic  operator involving a sum of mean oscillations taken over the sparse collection.
\begin{lemma}\label{SparseDom}
Let $Q_0$ be a cube in $\mathbb{R}^n$, $w\in A_\infty^d(Q_0)$ with constants $C_w>0$ and $\delta_w>0$, and $f\in L^1(Q_0)$. Then there is a collection $\mathcal{S}$ of dyadic cubes in $Q_0$ satisfying the following conditions:
\begin{enumerate}[label=\normalfont{(\alph*)}]
\item \label{sparse} There is a constant $\eta=\eta(C_w,\delta_w)>0$ and a collection $\left\{E_Q\mathbin{:} Q\in\mathcal{S}\right\}$ of pairwise disjoint  sets such that for every $Q\in\mathcal{S}$, $E_Q$ is a measurable subset of $Q$ with $w(E_Q)\geq \eta w(Q)$. 
\item \label{joku} For every Lebesgue point $x\in Q_0$ of $f$, we have
\begin{equation}\label{claima}
\abs{f(x)-f_{Q_0}}\leq C\sum_{Q\in\mathcal{S}}\Chi_Q(x)\dashint_Q\abs{f(y)-f_Q}\dd y
\end{equation}
with $C=C(n, C_w,\delta_w)>0$. 
\end{enumerate}
\end{lemma}
\begin{remark}Condition \ref{sparse} means that $\mathcal{S}$ is a sparse collection of cubes with respect to the weight $w$. In \ref{joku}, recall that almost every point of $Q_0$ is a Lebesgue point of $f$.
\end{remark}
\begin{proof} 
 Lemma \ref{SparsePre} provides us with a collection $\mathcal{S}\subset \mathcal{D}(Q_0)$. We are going to construct the family $\{E_Q:Q\in\mathcal{S}\}$ by removing selected parts of the cubes $S\in\mathcal{S}$, namely their $\mathcal{S}$-children. For every $S\in\mathcal{S}$, let
\begin{equation}\label{ES}
E_S = S\setminus\bigcup_{S'\in\ch_\mathcal{S}(S)}S'.
\end{equation}
We show that $\left\{E_S\mathbin{:}S\in\mathcal{S}\right\}$ is the collection of pairwise disjoint sets postulated by \ref{sparse}. To prove disjointness, fix $S, R\in\mathcal{S}$ such that $S\neq R$. If, $S\cap R = \emptyset$, then clearly $E_S\cap E_R = \emptyset$. If, say, $R\subset S$, there is a cube $S'\in\ch_\mathcal{S}(S)$ such that $R\subset S'$ and therefore $E_R\cap E_S \subset S'\cap E_S = \emptyset$. Hence $\left\{E_S\mathbin{:}S\in\mathcal{S}\right\}$ is a collection of pairwise disjoint sets.

We still need to show that $w(E_S)\geq\eta w(S)$ for a fixed $S\in\mathcal{S}$. By \eqref{S3}, it holds that 
\[ \sum_{S'\in\ch_\mathcal{S}(S)}w(S')\leq C_w\rho^{-\delta_w}w(S). \]
Furthermore, the collection $\ch_\mathcal{S}(S)$ is pairwise disjoint, so we have 
\begin{equation*}
w(E_S) = w(S)-\sum_{S'\in\ch_\mathcal{S}(S)}w(S') \geq \eta w(S),
\end{equation*}
where $\eta = 1-C_w\rho^{-\delta_w} > 0$ (recall that $C_w\rho^{-\delta_w}<1$). This completes the proof of condition \ref{sparse}. 

To prove that condition \ref{joku} holds, we introduce the following ``dyadic difference'' operator:
\begin{equation*}
\Delta_Qf(x) = \sum_{Q'\in\ch_\mathcal{D}(Q)}\Chi_{Q'}(x)(f_{Q'}-f_Q),
\end{equation*}
where $Q$ is a $Q_0$-dyadic cube and $\ch_\mathcal{D}(Q)$ the collection of its $2^n$ dyadic children. To begin with, we fix a Lebesgue point $x\in Q_0$ of $f$ to estimate the left-hand side of \eqref{claima} by telescoping in terms of these dyadic differences:
\begin{align}\label{xsplit} 
\nonumber \abs{f(x)-f_{Q_0}} \Chi_{Q_0}(x) & = \abs{\sum_{Q\in\mathcal{D}(Q_0)}\Delta_Qf(x)} = \abs{\sum_{S\in\mathcal{S}}\sum_{Q\mathbin{:}\pi_\mathcal{S}Q=S}\Delta_Qf(x)}\\
& \leq \sum_{S\in\mathcal{S}}\abs{\sum_{Q\mathbin{:}\pi_\mathcal{S}Q=S}\Delta_Qf(x)}.
\end{align}
Fix now a $S\in\mathcal{S}$ and split the innermost sum with respect to the set $E_S$ defined by \eqref{ES}: 
\begin{equation*}
\sum_{Q\mathbin{:}\pi_\mathcal{S}Q=S}\Delta_Qf(x) =  \Chi_{S\setminus E_S}(x) \sum_{Q\mathbin{:}\pi_\mathcal{S}Q=S}\Delta_Qf(x) + \Chi_{E_S}(x) \sum_{Q\mathbin{:}\pi_\mathcal{S}Q=S}\Delta_Qf(x). 
\end{equation*}
We estimate each sum separately; the aim is to control each in terms of $\kappa(S) = \dashint_S\abs{f(x)-f_S}\dd x$. Beginning with the first one, we obtain by telescoping 
\begin{align*}
\Chi_{S\setminus E_S}(x) & \sum_{Q\mathbin{:}\pi_\mathcal{S}Q=S}\Delta_Qf(x) = \sum_{S'\in\ch_\mathcal{S}(S)}\sum_{Q\mathbin{:}\pi_\mathcal{S}Q = S}\Chi_{S'}(x)\Delta_Qf(x)\\
 & = \sum_{S'\in\ch_\mathcal{S}(S)}\sum_{Q\mathbin{:}\pi_\mathcal{S}Q = S}\sum_{Q'\in\ch_\mathcal{D}Q}\Chi_{S'}(x)\Chi_{Q'}(x)\left(f_{Q'}-f_Q\right).
\end{align*} 
Here, the first step follows by the fact that $S\setminus E_S = \cup_{S'\in\ch_\mathcal{S}(S)}S'$. Depending on $x\in S$ there is a unique ``tower'' of cubes, beginning from $S$ and down to the dyadic parent of $S'$, which is the unique $\mathcal{S}$-child of $S$ such that $x\in S'$. As a result, the ensuing double sum is a telescope, resulting in
\begin{equation*}
\Chi_{S\setminus E_S}(x) \sum_{Q\mathbin{:}\pi_\mathcal{S}Q=S}\Delta_Qf(x) = \sum_{S'\in\ch_\mathcal{S}(S)}\Chi_{S'}(x)(f_{S'}-f_S). 
\end{equation*}

Fix $S'\in\ch_\mathcal{S}(S)$ and let $\pi S'$ denote the dyadic parent of $S'$. Notice that $\pi_\mathcal{S}(\pi S') = S\in\mathcal{S}$, so we may use the property \eqref{S1} to estimate 
\begin{align}\label{xsplit1}
\nonumber \abs{f_{S'}-f_S} & = \abs{\kern2pt\dashint_{S'}\left(f(x)-f_S\right)\dd x} \leq \dashint_{S'}\abs{f(x)-f_S}\dd x\\
\nonumber & \leq 2^n\dashint_{\pi S'}\abs{f(x)-f_S}\dd x\\
& \leq \rho 2^n\dashint_S\abs{f(x)-f_S}\dd x = \rho 2^n\kappa(S). 
\end{align}
As for the second sum, fix $x\in E_S$ and let $\left(Q_k\right)_{k\in\mathbb{N}}$ be a sequence of dyadic cubes such that $x\in Q_k\subset S$ for all $k\in\mathbb{N}$, and $\abs{Q_k}\rightarrow 0$ as $k\rightarrow\infty$. Thus $\left(Q_k\right)_{k\in\mathbb{N}}$ shrinks nicely to $x$. Since $x$ is a Lebesgue point of $f$, we  may telescope
\begin{equation*}
\Chi_{E_S}(x) \sum_{Q\mathbin{:}\pi_\mathcal{S}Q=S}\Delta_Qf(x) = \Chi_{E_S}(x)(f(x)-f_S) =\lim_{k\rightarrow\infty} (f_{Q_k}-f_S).
\end{equation*}
Since $x\in Q_k\cap E_S$ and $Q_k\subset S$, we have $\pi_\mathcal{S}(Q_k) = S$ for every $k\in\mathbb{N}$. Hence, by the property \eqref{S1},
\begin{align*}
\abs{f_{Q_k}-f_S} & = \abs{\dashint_{Q_k}\left(f(x)-f_S\right)\dd x} \leq \dashint_{Q_k}\abs{f(x)-f_S}\dd x\\
& \leq \rho\dashint_S\abs{f(x)-f_S}\dd x = \rho\kappa(S).
\end{align*} 
Since this estimate is uniform with respect to $k$, we conclude that 
\begin{equation}\label{xsplit2}
\abs{f(x)-f_S} \leq \rho\kappa(S). 
\end{equation}
Collecting the estimates \eqref{xsplit}, \eqref{xsplit1}, and \eqref{xsplit2}, we find that
\begin{align*}
\abs{f(x)-f_{Q_0}}\Chi_{Q_0} & \leq \sum_{S\in\mathcal{S}}\abs{\sum_{Q\mathbin{:}\pi_\mathcal{S}Q=S}\Delta_Qf(x)}\\
& = \sum_{S\in\mathcal{S}}\left(\sum_{S'\in\ch_\mathcal{S}(S)}\Chi_{S'}(x)\rho 2^n\kappa(S) + \Chi_{E_S}(x)\rho\kappa(S)\right)\\
& \leq \rho 2^n \sum_{S\in\mathcal{S}}\Chi_S(x)\kappa(S), 
\end{align*} 
which concludes the proof.
\end{proof}

\section{Local Fefferman--Stein inequality}

The first sparse domination lemma and a duality argument lead to a result for the dyadic sharp maximal function,
which is a variant of the Fefferman--Stein inequality (see \cite{hardies}, Chapter III). This result is of independent interest.
\begin{dfn} For a cube $Q_0$ and $f\in L^1(Q_0)$, we define \emph{the dyadic sharp maximal function} by 
\begin{equation*}
M_{Q_0}^{d,\sharp}f(x) = \sup_{\substack{Q\subset Q_0 \\ x\ni Q}} \dashint_Q\abs{f(y)-f_Q}\dd y, 
\end{equation*}
where the supremum is taken over all dyadic cubes $Q\subset Q_0$ such that $x\in Q$.
\end{dfn} 
The following is a localized and weighted variant of the Fefferman--Stein inequality \cite{MR0447953}; see also Theorem III.3 in \cite{hardies}. The proof relies on the fact that dyadic sparse operators can be controlled by duality and maximal function arguments due to existence of the disjoint sets $\{E_S\,:\, S\in\mathcal{S}\}$ associated with the sparse collection of cubes $\mathcal{S}$.

\begin{theorem}\label{maxima}
Let  $Q_0\subset\mathbb{R}^n$ a cube,  $1<p<\infty$, $w\in A_\infty^d(Q_0)$, and $f\in L^1(Q_0)$. Then
\begin{equation*}
\int_{Q_0}\abs{f(x)-f_{Q_0}}^pw(x)\dd x \leq C\int_{Q_0} \left(M^{d,\sharp}_{Q_0}f(x)\right)^p w(x)\dd x,
\end{equation*}
where $C = C(n, p, C_w, \delta_w) > 0$. Here $C_w$ and $\delta_w$ are the $A_\infty^d(Q_0)$ constants for $w$. 
\end{theorem}
\begin{proof} Fix a $f\in L^1(Q_0)$ and let $\mathcal{S}$ be the associated collection of cubes given by the sparse domination lemma \ref{SparseDom}. On $Q_0$, define the function
\begin{equation*}
\widetilde{f}(x) = \sum_{Q\in\mathcal{S}}\Chi_Q(x)\dashint_Q\abs{f(y)-f_Q}\dd y.
\end{equation*}
By the sparse domination lemma \ref{SparseDom}, we have
\begin{equation*}
\int_{Q_0}\abs{f(x)-f_{Q_0}}^pw(x)\dd x \leq C(n, p, C_w, \delta_w)\int_{Q_0} \widetilde{f}(x)^pw(x)\dd x.
\end{equation*}
We estimate the $p$th root of the last integral by duality. Namely, it is enough to show that there is a constant $C=C(\eta,p)$ such that
\begin{equation*}
\abs{\int_{Q_0}\widetilde{f}(x)g(x)w(x)\dd x} \leq C \left(\int_{Q_0} \left(M^{d,\sharp}_{Q_0}f(x)\right)^pw(x)\dd x\right)^{\frac{1}{p}}
\end{equation*}
for every bounded measurable function $g$ in $Q_0$ with $\norm{g}_{L^q(Q_0; w\dd x)} = 1$ and $1/p+1/q=1$. Fix such a function $g$; we have \begin{equation}\label{xkertis}
\abs{\int_{Q_0}\widetilde{f}(x)g(x)w(x)\dd x} = \sum_{Q\in\mathcal{S}}\dashint_Q\abs{f(y)-f_Q}\dd y \abs{\int_{Q_0}\Chi_Q(x)g(x)w(x)\dd x}.
\end{equation}
Again thanks to the sparse domination lemma \ref{SparseDom} \ref{sparse}, we may estimate for each $Q\in\mathcal{S}$ 
\begin{align}\label{xxkertis}
\nonumber 
\abs{\int_{Q_0}\Chi_Q(x)g(x)w(x)\dd x} &=\abs{\int_Qg(x)w(x)\dd x} \\
& \leq w(Q)\cdot\frac{1}{w(Q)}\int_Q\abs{g(x)}w(x)\dd x \\
\nonumber &\leq \eta^{-1}w(E_Q)\cdot\frac{1}{w(Q)}\int_Q \lvert g(x)\rvert w(x)\dd x.
\end{align}
Combine \eqref{xkertis} and \eqref{xxkertis}, and apply Lemma \ref{strong}:
\begin{align*}
& \abs{\int_{Q_0}\widetilde{f}(x)g(x)w(x)\dd x}\\
& \hspace*{3em} \leq \eta^{-1}\sum_{Q\in\mathcal{S}}w(E_Q) \dashint_Q\abs{f(y)-f_Q}\dd y \cdot \frac{1}{w(Q)}\int_Q \lvert g(x)\rvert w(x)\dd x \\
& \hspace*{3em} \leq C(\eta)\sum_{Q\in\mathcal{S}}\int_{E_Q}M^{d,\sharp}_{Q_0}f(x) M^{d, w}_{Q_0}g(x)w(x)\dd x \\
& \hspace*{3em} \leq C\int_{Q_0}M^{d,\sharp}_{Q_0}f(x) M^{d, w}_{Q_0}g(x) w(x)\dd x\\
& \hspace*{3em} \leq C\norm{M^{d,\sharp}_{Q_0}f}_{L^p(Q_0;w\dd x)}\norm{M^{d,w}_{Q_0}g}_{L^q(Q_0;w\dd x)}\\
& \hspace*{3em} \leq C(\eta, p)\norm{g}_{L^q(Q_0; w\dd x)}\cdot \left(\int_{Q_0}\left(M^{d,\sharp}_{Q_0}f(x)\right)^p w(x)\dd x\right)^\frac{1}{p},
\end{align*}
which, recalling that $\norm{g}_{L^q(Q_0; w\dd x)} = 1$, is the desired result.
\end{proof}

In the proof of the two-weight inequality (Theorem \ref{lokaali}), we need to estimate the dyadic sharp maximal function in terms of the dyadic fractional maximal function of the gradient. To this end, we apply the following $(q,p)$-Poincaré inequality on cubes, which we state without proof. See \cite{MR1814364}, p.~164, for details. 
\begin{lemma}\label{morePoinca} Let $Q\subset \mathbb{R}^n$ be a cube and $1\le p<\infty$. Let $1\le q\le \frac{np}{n-p}$ for $1\le p<n$, 
and  $1\le q<\infty$ for $n\le p<\infty$. Then there is a constant $C = C(n,p,q)$ such that for all $u\in\Lip(Q)$
\begin{equation*}
\left(\dashint_Q\abs{u(x)-u_Q}^q\dd x\right)^\frac{1}{q}\leq Cl(Q)\left(\dashint_Q\abs{\grad u(x)}^p\dd x\right)^\frac{1}{p}. 
\end{equation*}\qed
\end{lemma}
\begin{lemma}\label{maxim}
Let $Q_0\subset\mathbb{R}^n$ a cube and $u\in \Lip(Q_0)$. Then for all $x\in Q_0$
\begin{equation*}
M_{Q_0}^{d,\sharp}u(x) \leq C(n)M^d_{1, Q_0}\abs{\grad u(x)}. 
\end{equation*}
\end{lemma}
\begin{proof}
Fix a $x\in Q_0$ and a dyadic cube $Q\subset Q_0$ containing $x$. By the $(1,1)$-Poincaré inequality on cubes (Lemma \ref{morePoinca})
\begin{align*}
\dashint_Q\abs{u(y)-u_Q}\dd y & \leq C(n)l(Q)\dashint_Q\abs{\grad u(x)}\dd x\\
& = C(n)\frac{\abs{Q}^{1/n}}{\abs{Q}}\int_Q\abs{\grad u(x)}\dd x\\
& = C(n)\frac{1}{\abs{Q}^{1-1/n}}\int_Q\abs{\grad u(x)}\dd x\\
& \le C(n)M^d_{1, Q_0}\abs{\grad u(x)}.
\end{align*} 
 The final inequality follows since the choice of $Q$ is arbitrary.
\end{proof}

\section{Sparse domination II}\label{spardo}

The following lemma is the other of the two sparse domination results we need. The idea is from Pérez; see the proof of Theorem 1.1 in \cite{MR1052009}. 
\begin{lemma}\label{inequalemma}Let $0\le \alpha<n$, $Q_0$ be a cube in $\mathbb{R}^n$, and $\sigma\in A_\infty^d(Q_0)$ with constants $C_\sigma>0$, $\delta_\sigma>0$. For every $f\in L^1(Q_0,)$, there is a collection $\mathcal{S}$ of $Q_0$-dyadic cubes satisfying the following conditions: 
\begin{enumerate}[label=\normalfont{(\alph*)}]
\item \label{spars} There is a constant  $\eta=\eta(C_\sigma,\delta_\sigma)>0$  and a collection $\left\{E_Q\mathbin{:} Q\in\mathcal{S}\right\}$ of pairwise disjoint sets such that for every $Q\in\mathcal{S}$, $E_Q$ is a measurable subset of $Q$ with $\sigma(E_Q)\geq \eta\sigma(Q)$. 
\item \label{sparsineq} For almost every  $x\in Q_0$ and every $1\leq p<\infty$, we have 
\begin{equation*}
\left(M^d_{\alpha, Q_0}f(x)\right)^p \leq C\sum_{Q\in\mathcal{S}}\Chi_Q(x)\left(\frac{1}{\abs{Q}^{1-\alpha/n}}\int_Q\abs{f(y)}\dd y\right)^p
\end{equation*}
with $C= C(n, p, C_\sigma,\delta_\sigma) >0$.
\end{enumerate}
\end{lemma} 
\begin{proof}For simplicity, we take $f$ to be nonnegative. To begin with, fix a constant $a>2^n$ such that 
\begin{equation*}
C_\sigma\left(\frac{2^n}{a}\right)^{\delta_\sigma}<1. 
\end{equation*}
Analogously with the proof of Lemma \ref{SparseDom}, let us denote $\rho = a\cdot 2^{-n} > 1$. We may assume that
\begin{equation*}
\frac{1}{\abs{Q_0}^{1-\alpha/n}}\int_{Q_0}f(y)\dd y > 0.
\end{equation*} 
If this is not the case, then $f = 0$ almost everywhere in $Q_0$, which implies $M_{\alpha,Q_0}^d f = 0$ everywhere in $Q_0$. Hence there is nothing to estimate and we may choose $\mathcal{S}=\left\{Q_0\right\}$ and $E_{Q_0}=Q_0$. 

Let $k_0$ be the smallest integer satisfying
\begin{equation}\label{faq}
\frac{1}{\abs{Q_0}^{1-\alpha/n}}\int_{Q_0}f(y)\dd y \leq a^{k_0}. 
\end{equation}
For each $k>k_0$, denote
\begin{align*}
S_k & = \{x\in Q_0\mathbin{:} a^k < M^d_{\alpha, Q_0}f(x)\}. 
\end{align*}
Let $\mathcal{S}_{k_0} = \left\{Q_0\right\}$ and, for $k>k_0$, we let $\mathcal{S}_k$ denote the collection of maximal $Q_0$-dyadic cubes $Q\subset Q_0$ satisfying
\begin{equation}\label{stopp}
a^k < \frac{1}{\abs{Q}^{1-\alpha/n}}\int_Qf(y)\dd y. 
\end{equation}
Since $k>k_0$, we find that each maximal cube is strictly contained in $Q_0$. Observe also that $S_k = \cup_{Q\in\mathcal{S}_k}Q$ if $k>k_0$. By the nestedness of the dyadic structure and the construction above, for each $k\geq k_0$ and $R\in\mathcal{S}_{k+1}$, there exists a unique $Q\in\mathcal{S}_k$ such that $R\subset Q$. For each $k\geq k_0$ and $Q\in\mathcal{S}_k$, let
\begin{equation}\label{EKQ}
E_{k,Q} = Q\setminus\bigcup_{R\in\mathcal{S}_{k+1}}R = Q\setminus\bigcup_{\substack{R\in\mathcal{S}_{k+1}\\R\subset Q}} R.
\end{equation}
To verify \ref{spars}, we need to prove that the inequality
\begin{equation}\label{goal}
\sigma(E_{k,Q})\geq (1-C_\sigma\rho^{-\delta_\sigma})\sigma(Q)
\end{equation}
holds for all $Q\in\mathcal{S}_k$ and $k\geq k_0$. Fix $k\ge k_0 $ and $Q\in\mathcal{S}_k$.
By \eqref{EKQ}, the collection $\left\{E_{k,Q}\mathbin{:}k\geq k_0, Q\in\mathcal{S}_k\right\}$ is pairwise disjoint. To begin with, let $k> k_0$. Recall that $\pi Q$ denotes the dyadic parent of $Q$. It follows from the stopping construction that 
\begin{equation*}
\frac{1}{\abs{\pi Q}^{1-\alpha/n}}\int_{\pi Q}f(y)\dd y \leq a^k,
\end{equation*}
because the collection $\mathcal{S}_k$ is maximal with respect to the stopping rule \eqref{stopp}. Using this and the fact that $\alpha\geq 0$, we obtain 
\begin{align}\label{rs}
\nonumber \sum_{\substack{R\in \mathcal{S}_{k+1}\\R\subset Q}}\abs{R} & \leq \sum_{\substack{R\in \mathcal{S}_{k+1}\\R\subset Q}} \frac{\abs{R}^{\alpha/n}}{a^{k+1}}\int_Rf(y)\dd y \leq \frac{\abs{Q}^{\alpha/n}}{a^{k+1}} \sum_{\substack{R\in \mathcal{S}_{k+1}\\R\subset Q}} \int_Rf(y)\dd y\\
\nonumber & \leq \frac{\abs{Q}^{\alpha/n}}{a^{k+1}}\int_Qf(y)\dd y \leq \frac{\abs{Q}}{a^{k+1}}\cdot\frac{\left(2^n\right)^{1-\alpha/n}}{\abs{\pi Q}^{1-\alpha/n}}\int_{\pi Q}f(y)\dd y \\ 
& \leq \frac{2^n}{a}\abs{Q}.
\end{align}

As for the case $k=k_0$, recall that $k_0$ was chosen as the smallest integer such that \eqref{faq} holds, and $\mathcal{S}_{k_0} = \left\{Q_0\right\}$. Again applying the stopping rule \eqref{stopp} in the first inequality below, we have 
\begin{align}\label{rs2}
\nonumber\sum_{\substack{R\in \mathcal{S}_{k_0+1}\\R\subset Q_0}}\abs{R} & \nonumber\leq \sum_{\substack{R\in \mathcal{S}_{k_0+1}\\R\subset Q_0}} \frac{\abs{R}^{\alpha/n}}{a^{k_0+1}}\int_Rf(y)\dd y \leq \frac{\abs{Q_0}^{\alpha/n}}{a^{k_0+1}} \sum_{\substack{R\in \mathcal{S}_{k_0+1}\\R\subset Q_0}} \int_Rf(y)\dd y \\
& \nonumber\leq \frac{\abs{Q_0}^{\alpha/n}}{a^{k_0+1}}\int_{Q_0}f(y)\dd y \leq \frac{\abs{Q_0}}{a}\cdot\frac{1}{a^{k_0}}\cdot\frac{1}{\abs{Q_0}^{1-\alpha/n}}\int_{Q_0}f(y)\dd y\\
& \leq \frac{1}{a}\abs{Q_0}.
\end{align}

Combining the $A_\infty^d(Q_0)$ property of $\sigma$ with the estimates \eqref{rs} and \eqref{rs2}, we obtain
\begin{align*}
\sigma\left(\bigcup_{\substack{R\in\mathcal{S}_{k+1}\\R\subset Q}}R\right)\sigma(Q)^{-1} & 
\leq C_\sigma\left(\Biggl\lvert\bigcup_{\substack{R\in\mathcal{S}_{k+1}\\R\subset Q}}R\;\Biggr\rvert\abs{Q}^{-1}\right)^{\delta_\sigma} \\
& \leq C_\sigma\left(\frac{2^n}{a}\right)^{\delta_\sigma} = C_\sigma\rho^{-\delta_\sigma}. 
\end{align*}
This lets us conclude that the inequality \eqref{goal} holds: 
\begin{equation*}
\sigma(E_{k,Q}) \geq \sigma(Q)-C_\sigma\rho^{-\delta_\sigma}\sigma(Q) = (1-C_\sigma\rho^{-\delta_\sigma})\sigma(Q). 
\end{equation*}
Looking back at the measure estimates \eqref{rs} and \eqref{rs2}, we notice that 
\begin{equation*}
\sum_{\substack{R\in \mathcal{S}_{k+1}\\R\subset Q}}\abs{R} \leq \frac{2^n}{a}\abs{Q} < \abs{Q}.
\end{equation*}
In other words, each $R\in\mathcal{S}_{k+1}$ that is contained in $Q\in\mathcal{S}_k$ is strictly smaller than $Q$ itself. Each $Q\in\mathcal{S}$ belongs to a unique collection $\mathcal{S}_k$, and we may define another collection of cubes $\mathcal{S}=\cup_{k\geq k_0}\mathcal{S}_k$ without including duplicates as well as identify $E_Q=E_{k,Q}$. Then, the condition \ref{spars} holds with $\eta = 1-C_\sigma\rho^{-\delta_\sigma} > 0$. 

It remains to prove \ref{sparsineq}. Consider for $k>k_0$ the sets
\begin{equation*}
D_k = \{x\in Q_0\mathbin{:} a^k < M^d_{\alpha, Q_0}f(x) \leq a^{k+1}\}.
\end{equation*}
We note that \[
Q_0\setminus \bigcup_{k>k_0} D_k = \left\{x\in Q_0\mathbin{:}M^d_{\alpha,Q_0} f(x)=\infty\right\} \bigcup \left\{x\in Q_0\mathbin{:} M^d_{\alpha, Q_0}f(x)\leq a^{k_0+1}\right\}.
\]  
The first set in the right-hand side is of zero measure, since $M^d_{\alpha,Q_0}f(x)\le l(Q_0)^\frac{\alpha}{n} M_{Q_0}^{d,1} f(x)$ for every $x\in Q_0$. Thus, inequality \eqref{weako} in Lemma \ref{strong} shows that  $M^d_{\alpha,Q_0}f(x) <\infty$ for almost every $x\in Q_0$. Furthermore, $k_0$ was chosen as the smallest integer such that \eqref{faq} holds, which means that 
\begin{equation*}
a^{k_0+1} = a^2\cdot a^{k_0-1} < a^2\cdot \frac{1}{\abs{Q_0}^{1-\alpha/n}}\int_{Q_0}f(y)\dd y. 
\end{equation*} 
Also, recall the stopping condition \eqref{stopp}: $\mathcal{S}_{k_0} = \left\{Q_0\right\}$, and for $k>k_0$, $\mathcal{S}_k$ is the collection of maximal $Q_0$-dyadic cubes $Q\subset Q_0$ satisfying
\begin{equation*}
a^k < \frac{1}{\abs{Q}^{1-\alpha/n}}\int_Qf(y)\dd y. 
\end{equation*}
Now fix a  $x\in Q_0\setminus\{x\in Q_0\mathbin{:}M^d_{\alpha,Q_0}f(x)=\infty\}$. With the preceding remarks, we have that 
\begin{align*}
\left(M^d_{\alpha, Q_0}f(x)\right)^p & = \left(M^d_{\alpha, Q_0}f(x)\right)^p\Chi_{Q_0\setminus \cup_k D_k}(x) + \sum_{k>k_0}\left(M^d_{\alpha, Q_0}f(x)\right)^p\Chi_{D_k}(x)\\
& \leq a^{(k_0+1)p}\Chi_{Q_0\setminus \cup_k D_k}(x) + a^p\sum_{k>k_0}a^{kp}\Chi_{S_k}(x)\\
& \leq a^{2p}\left(\frac{1}{\abs{Q_0}^{1-\alpha/n}}\int_{Q_0}f(y)\dd y\right)^p\Chi_{Q_0}(x) \\
& \hspace*{3em} + a^p\sum_{k>k_0}\sum_{Q\in \mathcal{S}_k}\left(\frac{1}{\abs{Q}^{1-\alpha/n}}\int_Qf(y)\dd y\right)^p\Chi_Q(x)\\
& \leq a^{2p}\sum_{k\geq k_0}\sum_{Q\in \mathcal{S}_k}\left(\frac{1}{\abs{Q}^{1-\alpha/n}}\int_Qf(y)\dd y\right)^p\Chi_Q(x)\\
& = a^{2p}\sum_{Q\in\mathcal{S}}\left(\frac{1}{\abs{Q}^{1-\alpha/n}}\int_Qf(y)\dd y\right)^p\Chi_Q(x).
\end{align*}
This concludes the proof of the lemma.
\end{proof}

The following two-weight inequality for the fractional maximal function is a localized variant of a result due to Pérez (\cite{MR1052009}, Theorem 1.1). 
\begin{theorem}\label{aplabel} Let $Q_0\subset\mathbb{R}^n$ be a cube. Furthermore, let $0\leq\alpha<n$, $1<p\leq q<\infty$, and $\left(v,w\right)$ a pair of weights in $Q_0$ such that $\sigma=v^{-1/(p-1)}\in A_\infty^d(Q_0)$. The following conditions are equivalent:
\begin{enumerate}[label=\normalfont{(\alph*)}]
\item\label{AA} There is a  
 $C >0$ 
such that, for all $f\in L^1(Q_0)$,
\begin{equation*}
\left(\int_{Q_0}\left(M^d_{\alpha, Q_0}f(x)\right)^q w(x)\dd x\right)^\frac{1}{q} \leq C\left(\int_{Q_0} \abs{f(x)}^pv(x)\dd x\right)^{\frac{1}{p}}. 
\end{equation*} 
\item\label{BB} There exists a $K>0$ such that, for all dyadic cubes $Q\subset Q_0$,
\begin{equation*}
\left(\frac{1}{\abs{Q}^{1-\alpha/n}}\right)^pw(Q)^{\frac{p}{q}}\sigma(Q)^{p-1}\leq K.
\end{equation*}
\end{enumerate}
In the implication from \ref{BB} to \ref{AA}, the constant $C$ is of the form $C(n,p,C_\sigma,\delta_\sigma,K)$. Here $C_\sigma$ and $\delta_\sigma$ are the $A_\infty^d(Q_0)$ constants for $\sigma$.
\end{theorem}
\begin{proof}
First we show that \ref{AA} implies \ref{BB}. Fix a dyadic cube $Q\subset Q_0$ and let $f = v^{-1/(p-1)}\Chi_Q$. 
As per the definition of the dyadic fractional maximal function, we clearly have
\begin{equation*}
\frac{1}{\abs{Q}^{1-\alpha/n}}\int_Q v(y)^{-\frac{1}{p-1}}\dd y \leq M^d_{\alpha, Q_0}f(x) \hspace*{2em}\text{ for every }x\in Q.
\end{equation*}
Now 
\begin{align}\label{okjpg}
\nonumber & {\abs{Q}}^{-(1-\alpha/n)q}\int_Q w(x)\dd x\; \left(\int_Q v(y)^{-\frac{1}{p-1}}\dd y\right)^q =\int_Q w(x)\dd x\;\left(\frac{1}{\abs{Q}^{1-\alpha/n}}\int_Q v(y)^{-\frac{1}{p-1}}\dd y\right)^q\\
& \hspace*{3em} \leq \int_{Q_0}\left(M^d_{\alpha, Q_0}f(x)\right)^qw(x)\dd x \leq C^q\left(\int_{Q_0}\abs{f(x)}^pv(x) \dd x\right)^\frac{q}{p}\\
\nonumber & \hspace*{3em} = C^q\left(\int_Q v(x)^{-\frac{1}{p-1}} \dd x\right)^\frac{q}{p},
\end{align}
where the second inequality on line \eqref{okjpg} follows from \ref{AA}. From here, we conclude that 
\begin{equation*}
\abs{Q}^{-(1-\alpha/n)p}\left(\int_Q w(x)\dd x\right)^\frac{p}{q}\left(\int_Q v(y)^{-\frac{1}{p-1}}\dd y\right)^{p-1}\leq C^p,
\end{equation*}
 that is,  $\abs{Q}^{-(1-\alpha/n)p}w(Q)^{p/q} \sigma(Q)^{p-1} \leq C^p$.

Next we show that \ref{BB} implies \ref{AA}. For a fixed $f\in L^1(Q_0)$ and $\sigma\in A_\infty^d(Q_0)$, let $\mathcal{S}$ be the associated collection of cubes given by the sparse domination lemma \ref{inequalemma}. Then, we have 
\begin{align}
\nonumber & \left(\int_{Q_0} \left(M^d_{\alpha,Q_0}f(x)\right)^q w(x)\dd x\right)^\frac{p}{q} = \norm{\left(M^d_{\alpha,Q_0}f\right)^p}_{L^{q/p}(Q_0;w\dd x)}\\
\label{line} &\hspace{3em} \leq C(n,p,C_\sigma,\delta_\sigma)\norm{\sum_{Q\in\mathcal{S}}\Chi_Q\left(\frac{1}{\abs{Q}^{1-\alpha/n}}\int_Q\abs{f(y)}\dd y\right)^p}_{L^{q/p}(Q_0;w\dd x)}\\
\nonumber &\hspace{3em} \leq C\sum_{Q\in\mathcal{S}}\norm{\Chi_Q\left(\frac{1}{\abs{Q}^{1-\alpha/n}}\int_Q\abs{f(y)}\dd y\right)^p}_{L^{q/p}(Q_0;w\dd x)}\\
\nonumber &\hspace{3em} = C\sum_{Q\in\mathcal{S}}\left(\frac{1}{\abs{Q}^{1-\alpha/n}}\int_Q\abs{f(y)}\dd y\right)^pw(Q)^\frac{p}{q},
\end{align}
where \eqref{line} follows from Lemma \ref{inequalemma}. On the other hand, by \ref{BB} we have that
\begin{align*}
& \sum_{Q\in\mathcal{S}}\left(\frac{1}{\abs{Q}^{1-\alpha/n}}\int_Q\abs{f(y)}\dd y\right)^pw(Q)^\frac{p}{q}\\
& \hspace{3em} = \sum_{Q\in\mathcal{S}}\left(\frac{1}{\abs{Q}^{1-\alpha/n}}\right)^pw(Q)^\frac{p}{q}\sigma(Q)^{p-1}\left(\frac{1}{\sigma(Q)}\int_Q\abs{f(y)}\sigma(y)^{-1}\sigma(y)\dd y\right)^p\sigma(Q)\\
& \hspace{3em} \leq K\sum_{Q\in\mathcal{S}}\left(\frac{1}{\sigma(Q)}\int_Q\abs{f(y)}\sigma(y)^{-1}\sigma(y)\dd y\right)^p\sigma(Q). 
\end{align*}
Recall from Lemma \ref{inequalemma} \ref{spars} that for $Q\in\mathcal{S}$, we have $\sigma(E_Q)\ge\eta\sigma(Q)$. Furthermore, Lemma \ref{inequalemma} states that $\left\{E_Q\right\}_{Q\in\mathcal{S}}$ is a collection of pairwise disjoint sets. Hence we may continue estimating
\begin{align}
\nonumber & K\sum_{Q\in\mathcal{S}}\left(\frac{1}{\sigma(Q)}\int_Q\abs{f(y)}\sigma(y)^{-1}\sigma(y)\dd y\right)^p\sigma(Q)\\
\nonumber & \hspace{3em} \leq \frac{K}{\eta}\sum_{Q\in\mathcal{S}}\left(\frac{1}{\sigma(Q)}\int_Q\abs{f(y)}\sigma(y)^{-1}\sigma(y)\dd y\right)^p\sigma(E_Q)\\
\nonumber & \hspace{3em} = \frac{K}{\eta}\sum_{Q\in\mathcal{S}}\int_{E_Q}\left(M^{d,\sigma}_{Q_0}\left(f\sigma^{-1}\right)(x)\right)^p\sigma(x)\dd x\\
\label{kert9} & \hspace{3em} \leq \frac{K}{\eta}\int_{Q_0}\left(M^{d,\sigma}_{Q_0}\left(f\sigma^{-1}\right)(x)\right)^p\sigma(x)\dd x\\
\label{kert99} & \hspace{3em} \leq C(p)\frac{K}{\eta}\int_{Q_0}\abs{f(x)\sigma(x)^{-1}}^p\sigma(x)\dd x. 
\end{align}
In \eqref{kert9}, we applied disjointness, while \eqref{kert99} follows from the boundedness of $M^{d, \sigma}_{Q_0}$ (Lemma \ref{strong}). This concludes the proof, since $v(x)=\sigma(x)^{-(p-1)}$ for every $x\in Q_0$. 
\end{proof}
\begin{remark}
The assumption $\sigma=v^{-\frac{1}{p-1}}\in A_\infty^d(Q_0)$ is needed in the implication from \ref{BB} to  \ref{AA} in Theorem \ref{aplabel}. In \cite{MR3606547} certain testing conditions are used in the case $\alpha=0$
to characterize a closely related boundedness result under weaker assumptions on the weight $\sigma$. Moreover, the special case $\alpha=0$ of Theorem~\ref{aplabel} is closely related to Theorem 1.15 in \cite{MR3606547}. 
\end{remark}

Our main local result is the following two-weight Poincaré inequality in $Q_0$, provided that the weights involved satisfy suitable $A_\infty$ conditions and the dyadic compatibility condition \eqref{koo0}. The theorem echoes an earlier result by Chua \cite{MR1140667}, while being more strictly localized. 
\begin{theorem}\label{lokaali} Let $Q_0$ be a cube in $\mathbb{R}^n$. Furthermore, let $1<p\leq q<\infty$, $(v, w)$ a pair of weights in $Q_0$ such that $w\in A_\infty^d(Q_0)$, and $\sigma = v^{-1/(p-1)}\in A_\infty^d(Q_0)$. Suppose that there exists a constant $K>0$ such that 
\begin{equation}\label{koo0}
\left(\frac{1}{\abs{Q}^{1-1/n}}\right)^pw(Q)^{\frac{p}{q}}\sigma(Q)^{p-1} \leq K 
\end{equation}
for all $Q_0$-dyadic cubes $Q\subset Q_0$. Then, the inequality 
\begin{equation*}
\left(\int_{Q_0}\abs{u(x)-u_{Q_0}}^qw(x)\dd x\right)^\frac{1}{q} \leq C\left(\int_{Q_0} \abs{\grad u(x)}^pv(x)\dd x\right)^{\frac{1}{p}}
\end{equation*}
holds for every $u\in\Lip(Q_0)$ with 
\begin{equation*}
C = C(n, p, q, K, C_w, C_\sigma, \delta_w, \delta_\sigma) > 0, 
\end{equation*}
where $C_w, \delta_w,$ and $C_\sigma, \delta_\sigma$ are the $A_\infty^d(Q_0)$ constants for $w$ and $\sigma$, respectively.
\end{theorem}
\begin{proof}
Let $u\in\Lip(Q_0)$ and $\left(v,w\right)$ as assumed. We first apply Theorem \ref{maxima} and Lemma \ref{maxim}, then  Theorem \ref{aplabel}:
\begin{align*}
\left(\int_{Q_0}\abs{u(x)-u_{Q_0}}^qw(x)\dd x\right)^\frac{1}{q} & \leq C\left(\int_{Q_0} \left(M^{d,\sharp}_{Q_0}u(x)\right)^q w(x)\dd x\right)^{\frac{1}{q}}\\
&\leq C\left(\int_{Q_0}\left(M^d_{1,Q_0}\abs{\grad u(x)}\right)^qw(x)\dd x\right)^\frac{1}{q}\\
& \leq C\left(\int_{Q_0}\abs{\grad u(x)}^pv(x)\dd x\right)^\frac{1}{p}. 
\end{align*}
\end{proof}

\section{From local to global}

We proceed to provide conditions for domains $\Omega$ such that there is a constant $C=C(n,p,\Omega)$ satisfying 
\begin{equation}\label{e.lg_demo}
\inf_{c\in\mathbb{R}} \int_\Omega \lvert f(x)-c \rvert^p\dd x
\le C \sum_{Q\in\mathcal{W}(\Omega)}
\int_{Q*} \lvert f(x)-f_{Q^*} \rvert^p\dd x
\end{equation}
for every $f\in L^1_{\loc}(\Omega)$. Theorem~\ref {gt:weak} provides a weighted variant of this local-to-global inequality under the assumption that $\Omega$ is a \emph{Boman domain}. This class of domains was introduced by Boman \cite{Bom}. It is known that a Euclidean domain $\Omega$ is a Boman domain if and only if it is a John domain \cite{MR1427074}. John domains are a more general class than Lipschitz domains, since they can have twisting cones. These classes of domains have been used extensively, for instance, in connection with Poincar\'e inequalities. Their relevance is covered in \cite{MR1427074}.

The inequality \eqref{e.lg_demo} provides a mechanism to bootstrap inequalities starting from  corresponding inequalities on cubes inside the domain. The proof is based on a chaining argument, and we adapt the rather well known argument developed in \cite{ino}; see also \cite{MR1140667}. For this purpose, we need to define chains.

\begin{dfn}\label{gt:chaindef} Let $\Omega$ be a bounded domain in $\mathbb{R}^n$ and consider Whitney cubes $Q\in\mathcal{W}(\Omega)$. We say that
\begin{equation*}
\mathcal{C}(Q) = \left(Q_0, \ldots, Q_k\right)   \subset \mathcal{W}(\Omega)
\end{equation*}
is a \emph{chain} in $\Omega$ joining $Q_0$ to $Q = Q_k$, if $Q_i\neq Q_j$ whenever $i\neq j$, and for each $j\in\left\{1,\ldots, k\right\}$ there exists a cube $R\subset Q_j^*\cap Q^*_{j-1}$ for which
\begin{equation*}
l(R) \geq C(n)\max\left\{l(Q_j^*), l(Q^*_{j-1})\right\}.
\end{equation*}
\end{dfn}
The collection $\left\{\mathcal{C}(Q)\mathbin{:} Q\in\mathcal{W}(\Omega)\right\}$ is called a \emph{chain decomposition} of $\Omega$, with a fixed Whitney cube $Q_0$ as the common starting point for all chains. The \emph{shadow} of a Whitney cube $R\in\mathcal{W}(\Omega)$ is the set $\mathcal{S}(R)=\left\{Q\in\mathcal{W}(\Omega)\mathbin{:} R\in\mathcal{C}(Q)\right\}$. Worth noticing is the duality of the concepts of chain and shadow: $R\in\mathcal{C}(Q)$ if and only if $Q\in\mathcal{S}(R)$.

We will assume throughout that $\Omega$ is a Boman domain, which means that it satisfies the following chain condition.
\begin{dfn} A domain $\Omega\subset\mathbb{R}^n$ is said to satisfy the \emph{Boman chain condition} with constant $N\geq 1$ if there exists a chain decomposition of $\Omega$ such that  for all $R=Q(x_R,r_R)\in\mathcal{W}(\Omega)$
\begin{equation}\label{gt:bchain3} 
\bigcup_{Q\in\mathcal{S}(R)}Q\subset NR=Q(x_R,Nr_R).
\end{equation}
\end{dfn}
Open cubes, balls, and bounded Lipschitz domains are Boman domains in $\mathbb{R}^n$. More generally, so-called bounded John domains are examples of Boman domains; we refer to \cite{MR1427074} for details.

Finally, we require the weight $w$ to be doubling in the following sense. 
\begin{dfn} A weight  $w$ in an open set $\Omega$ is called \emph{doubling} in $\Omega$ with constant $D$ if there exists a constant $D = D(n,w)$ such that whenever $Q=Q(x_Q,r_Q)\subset\mathbb{R}^n$ is any cube with its midpoint $x_Q$ in $\Omega$  we have
\begin{equation*}
w\left(\Omega\cap Q(x_Q, 2r_Q)\right)\leq Dw\left(\Omega\cap Q(x_Q, r_Q)\right).
\end{equation*}
\end{dfn}

Next we show that doubling weights in $\mathbb{R}^n$ are doubling weights on Boman domains as well. 
%
\begin{lemma}\label{l.doubling}
Let $w$ be a doubling weight in $\mathbb{R}^n$ with constant $D_1$, and $\Omega$ a Boman domain with constant $N\ge 1$. Then, $w$ is doubling in $\Omega$ with constant $D(n,N,\Omega,Q_0,D_1)$.
\end{lemma}
\begin{proof}
Fix a  cube $Q=Q(x_Q,r_Q)$ with its midpoint $x_Q$ in $\Omega$. If $l(Q)> 2\diam(\Omega)$, then 
\[w(\Omega\cap 2Q)=w(\Omega)=w(\Omega\cap Q).\]
Hence, in the following we may assume that $l(Q)\le 2\diam(\Omega)$. It suffices to prove that there is a constant $\lambda=  \lambda(n,N,Q_0,\Omega)$ and another cube $R\subset Q\cap \Omega$ such that $l(Q)\leq \lambda l(R)$. Here $Q_0$ is the fixed cube in the chain decomposition of $\Omega$. Indeed, using this  and the global doubling property of $w$, we estimate
\begin{align*}
w(\Omega\cap 2Q) \leq w(2Q) \leq D_1w(Q) \leq C(D_1,\lambda)w(R) \leq C(D_1,\lambda) w(\Omega\cap Q).
\end{align*}

It now suffices to prove that the cube $R$ exists. Let $\rho=\rho(N,n)$ be such that $\rho (1+N)\sqrt n<\frac{1}{2}$. Fix a Whitney cube $P\in\mathcal{W}(\Omega)$ such that $x_Q\in P$. There are two cases to consider: either $l(P)>\rho l(Q)$ or not. In the first case, we take $R=Q(x_Q,\min\{r_Q,d(x_Q,\partial \Omega)/(2\sqrt n)\})\subset Q\cap \Omega$. Observe that $x_Q\in P\in\mathcal{W}(\Omega)$, and thus 
\[
d(x_Q,\partial \Omega)/(2\sqrt n)\ge d(P,\partial\Omega)/(2\sqrt n)\ge  l(P)/2> \rho l(Q)/2.
\]
Therefore $l(R)\ge C(\rho)l(Q)$.

Next assume that $l(P)\le \rho l(Q)$. Consider the chain $\mathcal{C}=(Q_0,\ldots,Q_k)$, where $Q_k=P$. Denote by $i_0$ the smallest index $i\in \{0,\ldots,k\}$ for which $l(Q_i)\le \rho l(Q)$ and denote $R=Q_{i_0}$. If $i_0=0$, then 
\[l(Q)\le 2\diam(\Omega)\le C(\Omega,Q_0)\diam(Q_0)=C(n,\Omega,Q_0)l(R).\]
On the other hand, if $i_0>0$, then
\[
l(Q)<\rho^{-1}l(Q_{{i_0}-1})\le C(n,\rho)l(Q_{i_0})=C(n,\rho)l(R).
\]
Here we also used the fact that adjacent cubes in the chain have comparable side lengths. Furthermore, we claim that $R\subset Q\cap \Omega$. Recall that $P\in\mathcal{S}(R)$ and thus $x_Q\in P\subset NR$ by the Boman chain condition \eqref{gt:bchain3}. Fix a $x\in R$. Then 
\begin{align*}
\abs{x-x_Q}&\le \abs{x-x_R}+\abs{x_R-x_Q}\le 
\diam(R)+\diam(NR)\\&=(1+N)\sqrt n l(R)\leq \rho (1+N)\sqrt n l(Q)<l(Q)/2.
\end{align*}
Hence $\abs{x-x_Q}<l(Q)/2$ and thus $x\in Q$. Since $R$ is a Whitney cube, it follows that $R\subset Q\cap \Omega$, as claimed.
\end{proof}

\begin{dfn}\label{gt:mutilde} Let $\Omega$ be an open set and let $w$ be a doubling weight in $\Omega$. We define the non-centered maximal function for $f\in L^1(\Omega; w\dd x)$ by
\begin{equation*}
\widetilde{M}^w f(x) = \sup_{Q\ni x}\frac{1}{w(\Omega\cap Q)}\int_{\Omega\cap Q}\abs{f(y)}w(y)\dd y,
\end{equation*}
where $x\in\Omega$ and the supremum is taken over cubes $Q\subset\mathbb{R}^n$ such that $x_Q\in\Omega$ and $Q\ni x$. 
\end{dfn}
We will make use of the fact that the maximal function $\widetilde{M}^w$ is bounded on $L^p(\Omega; w\dd x)$. For one instance of the proof, see Theorem 3.13 in \cite{MR2867756}. 
\begin{lemma}\label{gt:bddlemma}
Let $1<p<\infty$, $w$ a doubling weight in an open set $\Omega$, and $f\in L^p(\Omega; w\dd x)$. Then, $\widetilde{M}^w f\in L^p(\Omega; w\dd x)$ and there is a constant $C = C(n,p,w)$ such that
\begin{equation*}
\norm{\widetilde{M}^w f}_{L^p(\Omega; w\dd x)}\leq C\norm{f}_{L^p(\Omega; w\dd x)}. 
\end{equation*}
\end{lemma}

To use the following lemma is an idea of Iwaniec and Nolder's (\cite{ino}, Lemma 4). 
\begin{lemma}\label{gt:ino} Let $\Omega$ be a Boman domain with constant $N\geq 1$, $w$ a doubling weight in $\Omega$ with a constant $D$, and $1\leq p<\infty$. Furthermore, let $\left\{a_Q\mathbin{:}Q\in\mathcal{W}(\Omega)\right\}$ be a set of nonnegative real numbers. Then, there is a constant $C = C\left(p,D,N\right)$ for which 
\begin{equation*}
\norm{ \sum_{Q\in\mathcal{W}(\Omega)}a_Q\Chi_{\Omega\cap NQ}}_{L^p(\Omega; w\dd x)}\leq C\norm{\sum_{Q\in\mathcal{W}(\Omega)}a_Q\Chi_Q}_{L^p(\Omega; w\dd x)}. 
\end{equation*}
\end{lemma}
\begin{proof}
The case $p=1$ follows from the fact that the weight $w$ is doubling in $\Omega$; we will assume that $1<p<\infty$. By duality and the fact that bounded measurable functions are dense in $L^{p'}(\Omega; w\dd x)$, where $p'=p/(p-1)$, it is enough to show that
\begin{equation*}
\left\lvert\int_\Omega\sum_{Q\in\mathcal{W}(\Omega)}a_Q\Chi_{\Omega\cap NQ}(x)\psi(x)w(x)\dd x\right\rvert \leq C\norm{\sum_{Q\in\mathcal{W}(\Omega)}a_Q\Chi_Q}_{L^p(\Omega; w\dd x)}
\end{equation*}
for every bounded measurable function $\psi$ satisfying $\norm{\psi}_{L^{p'}(\Omega; w\dd x)}=1$. Let $\psi$ be such a function and $Q\in\mathcal{W}(\Omega)$. Then, for every $x\in \Omega\cap NQ$, 
\begin{equation*}
\int_{\Omega\cap NQ}\abs{\psi(y)}w(y)\dd y \leq w(\Omega\cap NQ)\widetilde{M}^w\psi(x).
\end{equation*}
Averaging this inequality over $Q\subset \Omega\cap NQ$ with respect to the measure $w\dd x$ and
using the fact that $w$ is doubling in $\Omega$, we obtain
\begin{align}\label{kertis}
\nonumber \int_{\Omega\cap NQ}\abs{\psi(y)}w(y)\dd y & \leq \frac{w(\Omega\cap NQ)}{w(Q)}\int_Q\widetilde{M}^w\psi(x)w(x)\dd x\\
& \leq C(D,N)\int_Q\widetilde{M}^w\psi(x)w(x)\dd x.
\end{align}
Using the triangle inequality and the estimate \eqref{kertis}, we obtain
\begin{align}
\nonumber \left\lvert\nonumber \int_\Omega\sum_{Q\in\mathcal{W}(\Omega)}a_Q\Chi_{\Omega\cap NQ}(x)\psi(x)w(x)\dd x\right\rvert & \leq \sum_{Q\in\mathcal{W}(\Omega)}a_Q\int_{\Omega\cap NQ}\abs{\psi(x)}w(x)\dd x\\
& \leq C(D,N)\sum_{Q\in\mathcal{W}(\Omega)}a_Q\int_Q\widetilde{M}^w\psi(x)w(x)\dd x.\label{this}
\end{align} 
Next, we rearrange \eqref{this} and apply Hölder's inequality: 
\begin{align*}
C\sum_{Q\in\mathcal{W}(\Omega)}a_Q\int_Q\widetilde{M}^w\psi(x)w(x)\dd x & = C\int_\Omega\sum_{Q\in\mathcal{W}(\Omega)}a_Q\Chi_Q(x)\widetilde{M}^w\psi(x)w(x)\dd x\\
& \leq C\norm{\sum_{Q\in\mathcal{W}(\Omega)}a_Q\Chi_Q}_{L^p(\Omega; w\dd x)}\norm{\widetilde{M}^w\psi}_{L^{p'}(\Omega; w\dd x)}.
\end{align*}
The desired result follows by the boundedness of the maximal function $\widetilde{M}^w$ (Lemma \ref{gt:bddlemma}) and the fact that $\norm{\psi}_{L^{p'}(\Omega; w\dd x)}=1$.
\end{proof}

Recall that $u_{w;Q} = w(Q)^{-1}\int_{Q}u(x)w(x)\dd x.$

\begin{lemma}\label{gt:chestimate}
Let $\Omega$ be a Boman domain, $w$ a doubling weight in $\Omega$ with a constant $D$,  and $\mathcal{C}(Q)=\left(Q_0, \ldots, Q_k\right)$ a chain joining the cube $Q_0$ to $Q_k=Q\in\mathcal{W}(\Omega)$, with $k$ depending on $Q$. Then, for all $u\in L^1_{\loc}(\Omega; w\dd x)$,
\begin{equation*}
\abs{u_{w; Q^*} - u_{w; Q_0^*}} \leq C(n,D)\sum_{R\in\mathcal{C}(Q)}\frac{1}{w(R^*)}\int_{R^*}\abs{u(x)-u_{w; R^*}}w(x)\dd x.
\end{equation*}
\end{lemma}
\begin{proof} Fix a $u\in L^1_{\loc}(\Omega; w\dd x)$. Then 
\begin{align}
\nonumber \abs{u_{w;Q^*}-u_{w;Q_0^*}} & = \abs{\sum_{i=1}^k u_{w;Q_i^*}-u_{w;Q_{i-1}^*}} \leq \sum_{i=1}^k\abs{u_{w;Q_i^*}-u_{w;Q_{i-1}^*}}\\
\label{suma} & \leq \sum_{i=1}^k\abs{u_{w; Q_i^*}-u_{w; Q_i^*\cap Q^*_{i-1}}} + \abs{u_{w; Q_i^*\cap Q^*_{i-1}} - u_{w; Q^*_{i-1}}}.
\end{align}
Let us fix $i=1,2,\ldots,k$. By the definition of a chain (Definition \ref{gt:chaindef}), there exists a cube $\widetilde{Q}\subset Q_i^*\cap Q_{i-1}^*$ such that $w(\widetilde{Q})>0$; likewise there is $\lambda$, depending only on the dimension $n$, such that $Q^*_{i-1}\cup Q_i^* \subset \lambda\widetilde{Q}$. Since the weight $w$ is doubling in $\Omega$, we obtain the estimate
\begin{equation}\label{xsomelabelz}
w(Q^*_i) \leq w(Q^*_{i-1}\cup Q_i^*) \leq w(\lambda\widetilde{Q}) \leq C(\lambda,D) w(\widetilde{Q}) \leq C(\lambda,D) w(Q_i^*\cap Q^*_{i-1}). 
\end{equation}
The exact same estimate holds for $w(Q^*_{i-1})$. We may estimate both parts of the sum \eqref{suma} as follows. For the sake of demonstration, choose the first one:
\begin{align}
\nonumber \abs{u_{w; Q_i^*}-u_{w; Q_i^*\cap Q^*_{i-1}}} & = \abs{\frac{1}{w(Q_i^*\cap Q_{i-1}^*)}\int_{Q_i^*\cap Q_{i-1}^*} \left(u(x)-u_{w; Q_i^*}\right)w(x)\dd x}\\
\nonumber & \leq \frac{1}{w(Q_i^*\cap Q_{i-1}^*)}\int_{Q_i^*\cap Q_{i-1}^*}\abs{u(x)-u_{w; Q_i^*}} w(x)\dd x\\
\label{ttupplaa} & \leq \frac{C(\lambda,D)}{w(Q_i^*)}\int_{Q_i^*}\abs{u(x)-u_{w; Q_i^*}}w(x)\dd x.
\end{align}
In \eqref{ttupplaa}, we applied the doubling property of $w$ through the estimate \eqref{xsomelabelz}. Estimating the second part of \eqref{suma} in like manner and taking all indices into account, we have
\begin{equation*}
\abs{u_{w;Q^*}-u_{w;Q_0^*}} \leq C(\lambda,D)\sum_{R\in\mathcal{C}(Q)}\frac{1}{w(R^*)}\int_{R^*}\abs{u(x)-u_{w; R^*}}w(x)\dd x,
\end{equation*}
which is the desired estimate, since $\lambda$ only depends on $n$.
\end{proof}

Finally, the following theorem connects the global scale and the cubewise estimates. 
\begin{theorem}\label{gt:weak}
Let $\Omega$ be a Boman domain with a constant $N\ge 1$ and $w$ a doubling weight in $\Omega$ with a constant $D$. If $u\in L^1_{\loc} (\Omega; w\dd x)$ and $1\le p<\infty$, then 
\begin{equation*}
\int_\Omega \abs{u(x)-u_{w; Q_0^*}}^pw(x)\dd x \leq C(n,p,D,N)\sum_{Q\in\mathcal{W}(\Omega)}\int_{Q^*}\abs{u(x)-u_{w; Q^*}}^pw(x)\dd x.
\end{equation*}
\end{theorem}
\begin{proof}
Let $Q_0$ be the fixed central cube in the chain decomposition of $\Omega$. By the triangle inequality for each $x\in\Omega$, we may write 
\begin{align*}
\abs{u(x)-u_{w; Q_0^*}} & \leq\abs{u(x) - \sum_{Q\in\mathcal{W}(\Omega)} u_{w; Q^{*}}\Chi_Q(x)} + \abs{\sum_{Q\in\mathcal{W}(\Omega)}u_{w; Q^{*}}\Chi_Q(x)-u_{w; Q_0^{*}}} \\
& = g_1(x) + g_2(x). 
\end{align*} 
Hence, it holds that 
\begin{equation*}
\int_\Omega\abs{u(y)-u_{w; Q_0^*}}^pw(x)\dd x \leq 2^p\int_\Omega g_1(x)^pw(x)\dd x + 2^p\int_\Omega g_2(x)^pw(x)\dd x.
\end{equation*}
We will estimate each integral on the right-hand side separately, beginning with the first one. Recalling that the Whitney cubes cover $\Omega$ and are disjoint, we have 
\begin{align*}
\int_\Omega g_1(x)^pw(x)\dd x & = \int_\Omega \abs{\sum_{Q\in\mathcal{W}(\Omega)}u(x)\Chi_Q(x) - \sum_{Q\in\mathcal{W}(\Omega)} u_{w; Q^{*}}\Chi_Q(x)}^pw(x)\dd x\\
&  = \sum_{Q\in\mathcal{W}(\Omega)}\int_Q\abs{u(x)-u_{w; Q^*}}^pw(x)\dd x\\
& \leq \sum_{Q\in\mathcal{W}(\Omega)}\int_{Q^*}\abs{u(x)-u_{w; Q^*}}^pw(x)\dd x,
\end{align*} 
which is of the required form. The integral associated with $g_2$ is estimated by using chains. We begin by 
\begin{align}
\nonumber \int_\Omega g_2(x)^pw(x)\dd x & = \int_\Omega\abs{\sum_{Q\in\mathcal{W}(\Omega)}\left(u_{w;Q^*}-u_{w;Q_0^*}\right)\Chi_Q(x)}^pw(x)\dd x\\
\label{wen}& \leq \int_\Omega\left(\sum_{Q\in\mathcal{W}(\Omega)}\abs{u_{w;Q^*}-u_{w;Q_0^*}}\Chi_Q(x)\right)^pw(x)\dd x.
\end{align}
Applying Lemma \ref{gt:chestimate} and Hölder's inequality, we obtain for every $Q\in\mathcal{W}(\Omega)$
\begin{equation}\label{termi}
\abs{u_{w;Q^*}-u_{w;Q_0^*}}\Chi_Q \leq C(n,D)\sum_{R\in\mathcal{C}(Q)}a_R\Chi_Q,
\end{equation}
where for every $R\in\mathcal{C}(Q)$
\begin{equation*}
a_R = \left(\frac{1}{w(R^*)}\int_{R^*}\abs{u(x)-u_{w;R^*}}^pw(x)\dd x\right)^\frac{1}{p} \geq 0.
\end{equation*}
Summing the estimates \eqref{termi} and using the shadow--chain duality, we obtain
\begin{align*}
\sum_{Q\in\mathcal{W}(\Omega)}\abs{u_{w;Q^*}-u_{w;Q_0^*}}\Chi_Q & \leq C(n,D)\sum_{Q\in\mathcal{W}(\Omega)}\sum_{R\in\mathcal{C}(Q)}a_R\Chi_Q\\
& = C\sum_{R\in\mathcal{W}(\Omega)}a_R\sum_{Q\in\mathcal{S}(R)}\Chi_Q.
\end{align*}
By the Boman chain condition \eqref{gt:bchain3}, we have $\sum_{Q\in\mathcal{S}(R)}\Chi_Q \leq \Chi_{NR}$, and
\begin{align*}
\sum_{Q\in\mathcal{W}(\Omega)} &\abs{u_{w;Q^*}-u_{w;Q_0^*}}\Chi_Q  = \left(\sum_{Q\in\mathcal{W}(\Omega)}\abs{u_{w;Q^*}-u_{w;Q_0^*}}\Chi_Q\right)\Chi_\Omega \\
&\leq C(n,D)\left(\sum_{R\in\mathcal{W}(\Omega)}a_R\Chi_{NR}\right)\cdot\Chi_\Omega = C(n,D)\left(\sum_{R\in\mathcal{W}(\Omega)}a_R\Chi_{\Omega\cap NR}\right).
\end{align*}
We substitute this back into \eqref{wen}, and respectively apply Lemma \ref{gt:ino}, Hölder's inequality for sums recalling that Whitney cubes are disjoint, and the doubling property of the weight $w$: 
\begin{align*}
& \int_\Omega g_2(x)^pw(x)\dd x \leq C(n,p,D)\int_\Omega\left(\sum_{R\in\mathcal{W}(\Omega)}a_R\Chi_{\Omega\cap NR}(x)\right)^pw(x)\dd x\\
& \hspace*{3em} \leq C(n,p,D,N)\int_\Omega\left(\sum_{R\in\mathcal{W}(\Omega)}a_R\Chi_R(x)\right)^pw(x)\dd x\\
& \hspace*{3em} = C\sum_{R\in\mathcal{W}(\Omega)}a_R^p\int_\Omega\Chi_R(x)w(x)\dd x\\
& \hspace*{3em} = C\sum_{R\in\mathcal{W}(\Omega)}\frac{w(R)}{w(R^*)}\int_{R^*}\abs{u(x)-u_{w;R^*}}^pw(x)\dd x\\
& \hspace*{3em} \leq C\sum_{R\in\mathcal{W}(\Omega)}\int_{R^*}\abs{u(x)-u_{w;R^*}}^pw(x)\dd x.
\end{align*}
\end{proof}

\section{Conclusion, applications to distance weights, and the $p$-Laplacian}

We are now ready to combine the local and local-to-global theorems into our main result. 
\begin{theorem}\label{main}
Let $\Omega\subset\mathbb{R}^n$ be a Boman domain with a constant $N\geq 1$, $1<p\leq q<\infty$, and $\left(v,w\right)$ a pair of weights in $\Omega$, $w$ doubling in $\Omega$ with a constant $D$, and $\sigma = v^{-1/(p-1)}$. Suppose that there exist strictly positive constants $C_w$ and $\delta_w$ such that for every cube $Q\in\mathcal{W}(\Omega)$ it holds that
\begin{equation*}
\frac{w(E)}{w(R)} \leq C_w\left(\frac{\abs{E}}{\abs{R}}\right)^{\delta_w}
\end{equation*}
for all $Q^*$-dyadic cubes $R\subset Q^*$ and all measurable sets $E\subset R$, and that there exist similar constants $C_\sigma$ and $\delta_\sigma$ for the weight $\sigma$. Furthermore, suppose that there exists a constant $K>0$ such that for every cube $Q\in\mathcal{W}(\Omega)$, we have
\begin{equation}
\left(\frac{1}{\abs{R}^{1-1/n}}\right)^pw(R)^{\frac{p}{q}}\sigma(R)^{p-1} \leq K 
\end{equation}
for all $Q^*$-dyadic cubes $R\subset Q^*$. Then, for any $u\in\Lip_{\loc}(\Omega)$
\begin{equation*}
\left(\inf_{c\in\mathbb{R}^n} \int_\Omega \abs{u(x)-c}^qw(x)\dd x\right)^\frac{1}{q} \leq C\left(\int_\Omega \abs{\grad u(x)}^pv(x)\dd x\right)^\frac{1}{p},
\end{equation*}
 where the constant $C = C(n, p, q, N, D, K, C_w, C_\sigma, \delta(w), \delta(\sigma)) > 0$. 
\end{theorem}
\begin{proof}
Begin by applying Theorem \ref{gt:weak}:
\begin{align}\label{gh}
\nonumber & \inf_{c\in\mathbb{R}^n} \int_\Omega \abs{u(x)-c}^qw(x)\dd x \leq C\sum_{Q\in\mathcal{W}(\Omega)}\int_{Q^*}\abs{u(x)-u_{w;Q^*}}^qw(x)\dd x \\
& \hspace*{3em} \leq C\sum_{Q\in\mathcal{W}(\Omega)} 2^{q-1} \left(\int_{Q^*}\abs{u(x)-u_{Q^*}}^qw(x)\dd x + w(Q^*)\abs{u_{Q^*}-u_{w;Q^*}}^q\right).
\end{align}
The second term can estimated using Hölder's inequality and absorbed into the first: 
\begin{align*}
& w(Q^*)\abs{u_{Q^*}-u_{w;Q^*}}^q = w(Q^*)\abs{\frac{1}{w(Q^*)}\int_{Q^*}\left(u_{Q^*}-u(x)\right)w(x)\dd x}^q\\
& \hspace*{3em} \leq w(Q^*)\left(\frac{1}{w(Q^*)}\int_{Q^*}\abs{u(x)-u_{Q^*}}^qw(x)\dd x\right) = \int_{Q^*}\abs{u(x)-u_{Q^*}}^qw(x)\dd x. 
\end{align*}
Continuing from \eqref{gh}, we apply Theorem \ref{lokaali}, the fact that $q\geq p$, and that the $Q^*$ have bounded overlap: $\sum_{Q\in\mathcal{W}}\Chi_{Q^*}\leq C(n)$. This yields
\begin{align*}
& C\sum_{Q\in\mathcal{W}(\Omega)} 2^{q}\int_{Q^*}\abs{u(x)-u_{Q^*}}^qw(x)\dd x \leq C\sum_{Q\in\mathcal{W}(\Omega)}\left(\int_{Q^*}\abs{\grad u(x)}^pv(x)\dd x\right)^\frac{q}{p}\\
& \hspace*{3em} \leq C \left(\sum_{Q\in\mathcal{W}(\Omega)}\int_{Q^*}\abs{\grad u(x)}^pv(x)\dd x\right)^\frac{q}{p} \leq C\left(\int_\Omega \abs{\grad u(x)}^pv(x)\dd x\right)^\frac{q}{p}.
\end{align*}
Taking $q$th roots completes the proof.
\end{proof}
It remains to say something about what pairs of weights  fulfill the requirements of Theorem \ref{main}. The following two theorems give two applications to distance weights, provided that the Aikawa (or, indeed, Assouad; see \cite{MR3055588}) dimension of the set from which the distance is measured is ``small enough''. The integral condition \eqref{aikawaa} below expresses exactly this, even if we will say no more about either Aikawa or Assouad. This being the case, it can be proven (see \cite{MR3900847}) that the distance function raised to a suitable power is in the class $A_\infty$, and we are able to apply the results at hand. 

The following result provides a localized variant of \cite{MR3900847}*{Theorem 6.1}.
\begin{theorem}\label{appl1}
Let $1<p\leq q\leq\frac{np}{n-p}<\infty$, and $E\subset\mathbb{R}^n$ a nonempty closed set satisfying 
\begin{equation}\label{aikawaa}
\dashint_{R(x_R,r)}d(x,E)^{-n+\frac{q}{p}(n-p)} \dd x\leq C_1 r^{-n+\frac{q}{p}(n-p)}
\end{equation}
for every cube $R$ with its midpoint $x_R$ in $E$ and $r>0$ with a constant $C_1=C_1(n,p,q,E)$. Let $Q\subset\mathbb{R}^n$ be a cube and $u\in\Lip(Q)$. Then there exists another constant $C = C(n,p,q,C_1)$ such that 
\begin{equation*}
\left(\int_Q\abs{u(x)-u_Q}^qd(x, E)^{-n+\frac{q}{p}(n-p)}\dd x\right)^\frac{1}{q} \leq C\left(\int_Q\abs{\grad u(x)}^p\dd x\right)^\frac{1}{p}.
\end{equation*}
\end{theorem}
\begin{proof}
Let $w(x) = d(x,E)^{-n+\frac{q}{p}(n-p)}$ and $v(x)= 1 = \sigma(x)$ for $x\in\mathbb{R}^n$. Because we have assumed condition \eqref{aikawaa}, Corollary 3.8 in \cite{MR3900847} tells that the weight $w$ belongs to the global Muckenhoupt classes $A_1(\mathbb{R}^n)\subset A_\infty(\mathbb{R}^n)$, whence in particular $w\in A^d_\infty(Q_0)$ with constants
$C_w$ and $\delta_w$ depending on $n$, $p$, $q$ and $C_1$ only. It is enough to show that there is a constant 
$K = K(n, p, q, C_1)$ such that \eqref{koo0} holds for all cubes $Q\in\mathcal{D}(Q_0)$; if so, the result follows from Theorem \ref{lokaali}. 

To this end, fix a dyadic cube $Q = Q(x_Q,r_Q)\in\mathcal{D}(Q_0)$. Assume first that $Q(x_Q, 2r_Q)\cap E \neq \emptyset$. Then there is a $z\in E$ such that $Q\subset Q(z,3r_Q)$. As we have assumed that $E$ satisfies the condition \eqref{aikawaa}, we may estimate
\begin{align*}
w(Q)^\frac{p}{q} & \leq \left(\int_{Q(z,3r_Q)}d(x, E)^{-n+\frac{q}{p}(n-p)}\dd x\right)^{\frac{p}{q}}\\
& \leq C(n,p,q,C_1)\cdot\left(r_Q^{\frac{q}{p}(n-p)}\right)^\frac{p}{q}\\
& = Cr_Q^{n-p} = C\abs{Q}^{1-\frac{p}{n}}.
\end{align*}
Since $\sigma(Q)^{p-1} = \abs{Q}^{p-1}$, we calculate
\begin{align*}
\left(\frac{1}{\abs{Q}^{1-\frac{1}{n}}}\right)^pw(Q)^\frac{p}{q}\sigma(Q)^{p-1} \leq C\abs{Q}^{-p+\frac{p}{n}+1-\frac{p}{n}+p-1} = C,
\end{align*}
which proves \eqref{koo0} in the case $Q(x_Q, 2r_Q)\cap E \neq \emptyset$. 
Assume next that $Q(x_Q,2r_Q)\cap E = \emptyset$. In this case, we have for every $x\in Q = Q(x_Q,r_Q)$
\begin{equation*}
C(n)d(x,E) \leq d(Q,E) \leq d(x,E).
\end{equation*}
Hence
\begin{align*}
w(Q)^\frac{p}{q}\sigma(Q)^{p-1} & \leq C(n,p,q)\abs{Q}^\frac{p}{q}d(Q,E)^{-\frac{np}{q}+n-p}\abs{Q}^{p-1}\\
& = C(n,p,q)\abs{Q}^{\frac{p}{q}+p-1}d(Q,E)^{-\frac{np}{q}+n-p}.
\end{align*}
By assumption $n-p-\frac{np}{q}\leq 0$ and $d(Q,E)\ge r_Q$, and thus 
\begin{align*}
\left(\frac{1}{\abs{Q}^{1-\frac{1}{n}}}\right)^pw(Q)^\frac{p}{q}\sigma(Q)^{p-1} & \leq C(n,p,q)\abs{Q}^{-p+\frac{p}{n}+\frac{p}{q}+p-1}d(Q,E)^{-\frac{np}{q}+n-p}\\
& \leq C\abs{Q}^{\frac{p}{n}+\frac{p}{q}-1}\abs{Q}^{\frac{1}{n}(-\frac{np}{q}+n-p)} = C,
\end{align*}
which completes the proof.
\end{proof}

\begin{theorem}
Let $1<p\leq q\leq\frac{np}{n-p}<\infty$, $\beta = n-\frac{q}{p}(n-p)$, and $\Omega\subset\mathbb{R}^n$ a Boman domain with constant $N\ge 1$ such that $w(x) = d(x,\partial\Omega)^{-\beta}$ is doubling in $\Omega$ with constant $D$.  
Then there is a constant $C = C(n,p,q,D,N)$ such that 
\begin{equation*}
\inf_{c\in\mathbb{R}}\left(\int_\Omega\abs{u(x)-c}^qd(x, \partial\Omega)^{-n+\frac{q}{p}(n-p)}\dd x\right)^\frac{1}{q} \leq C\left(\int_\Omega\abs{\grad u(x)}^p\dd x\right)^\frac{1}{p}
\end{equation*}
for every $u\in\Lip_{\loc}(\Omega)$. 
\end{theorem}
%
%
%
%
\begin{proof}
We can apply Theorem \ref{gt:weak}, and then pass from $u_{w;Q^*}$ to $u_{Q^*}$ as in the proof of Theorem \ref{main}:
\begin{align} 
\nonumber & \inf_{c\in\mathbb{R}}\int_\Omega\abs{u(x)-c}^qw(x)\dd x\\
\nonumber & \hspace*{3em} \leq C(n,q,D,N) \sum_{Q\in \mathcal{W}(\Omega)}\int_{Q^*}\abs{u(x)-u_{w; Q^*}}^qw(x)\dd x\\
\nonumber & \hspace*{3em} \leq C\sum_{Q\in\mathcal{W}(\Omega)}\int_{Q^*}\abs{u(x)-u_{Q^*}}^qw(x)\dd x\\
\label{ntj} & \hspace*{3em} \leq C\sum_{Q\in\mathcal{W}(\Omega)}l(Q^*)^{-\beta}\int_{Q^*}\abs{u(x)-u_{Q^*}}^q\dd x,
\end{align}
where $C = C(n,p,q,D,N)$. The final inequality \eqref{ntj} follows from \eqref{sideledil}, that is, the side length of a dilated  Whitney cube $Q^*$ is comparable to its distance from the domain boundary. 

To continue, fix a cube $Q\in\mathcal{W}(\Omega)$. The $(q,p)$-Poincaré inequality of Lemma \ref{morePoinca} implies 
\begin{align*}
& \int_{Q^*}\abs{u(x)-u_ {Q^*}}^q\dd x \leq C(n,p,q)l(Q^*)^q\abs{Q^*}^{1-\frac{q}{p}}\left(\int_{Q^*}\abs{\grad u(x)}^p\dd x\right)^\frac{q}{p}\\
& \hspace*{3em} \leq Cl(Q^*)^\beta\left(\int_{Q^*}\abs{\grad u(x)}^p\dd x\right)^\frac{q}{p},
\end{align*} 
since $q+n-\frac{nq}{p} = \beta$. Substituting this estimate back into \eqref{ntj}, keeping in mind that $q\geq p$ and that the dilated Whitney cubes $Q^*$ have a bounded overlap, we have
\begin{align*}
& \inf_{c\in\mathbb{R}}\int_\Omega\abs{u(x)-c}^qw(x)\dd x \leq C\sum_{Q\in\mathcal{W}(\Omega)}\left(\int_{Q^*}\abs{\grad u(x)}^p\dd x\right)^{\frac{q}{p}}\\
& \hspace*{3em} \leq C\left(\sum_{Q\in\mathcal{W}(\Omega)}\int_{Q^*}\abs{\grad u(x)}^p\dd x\right)^\frac{q}{p} \leq C\left(\int_{\Omega}\abs{\grad u(x)}^p\dd x\right)^\frac{q}{p},
\end{align*}
with $C = C(n,p,q,D,N)$, whereby the proof is complete. 
\end{proof}
\begin{remark}
As an alternative to the doubling assumption on $w$, we could have assumed $\Omega$ to be such that its boundary satisfies the Aikawa condition
\begin{equation}\label{aikawa}
\dashint_{R(x_R, r)}d(x,\partial\Omega)^{-n+\frac{q}{p}(n-p)}\dd x\leq C_1 r^{-n+\frac{q}{p}(n-p)}
\end{equation}
for every cube $R$ with its midpoint $x_R$ on $\partial\Omega$ and $r>0$ with a constant $C_1 = C_1(n,p,q,\Omega)$. In this case, as in the proof of Theorem~\ref{appl1}, we conclude that $w\in A_\infty(\mathbb{R}^n)$, and hence $w$ is doubling in $\mathbb{R}^n$. Lemma \ref{l.doubling} now implies that $w$ is doubling in $\Omega$ with constant $D=D(n,p,q,Q_0,\Omega)$.
\end{remark} 

Finally, let us take a look into the $p$-Laplace equation
\begin{equation}\label{plap}
\lap_p u = \div\left(\abs{\grad u}^{p-2}\grad u\right) = 0.
\end{equation}
More specifically, we consider weak supersolutions of the $p$-Laplace equation. Recall that $W^{1,p}_{\loc}(\Omega)$ is the Sobolev space of all $f\in L^p_{\loc}(\Omega)$ whose distributional first derivatives belong to $L^p_{\loc}(\Omega)$.

\begin{dfn}Let $\Omega\subset\mathbb{R}^n$ be an open set and $1<p<\infty$. We call $u\in W^{1,p}_{\loc}(\Omega)$ a \emph{weak supersolution} of the $p$-Laplace equation \eqref{plap} in $\Omega$ if for all nonnegative $\eta\in C^\infty_0(\Omega)$ 
\begin{equation*}
\int_\Omega \abs{\grad u(x)}^{p-2}\grad u(x)\cdot\grad\eta(x)\dd x \geq 0.
\end{equation*}
\end{dfn}
For instance, the first eigenfunction of the $p$-Laplacian is a nonnegative weak supersolution of the original equation. For an introduction to the eigenvalue problem, see \cite{MR2462954}. 

As per regularity theory, weak supersolutions could be said to satisfy half of the Harnack inequality. The following theorem is Theorem 3.59 in \cite{MR2305115} formulated for cubes. 
\begin{theorem}
Let $1<p<\infty$, $\Omega\subset\mathbb{R}^n$ be an open set, and let $u\in W^{1,p}_{\loc}(\Omega)$ be a nonnegative weak supersolution of the $p$-Laplace equation in $\Omega$. Then, for each $0<\beta<\infty$ with $\beta(n-p)<n(p-1)$ there is a constant $C = C(n,p,\beta)$ such that 
\begin{equation}\label{halfnack} 
\left(\dashint_Q  u(x)^\beta \dd x \right)^\frac{1}{\beta} \leq C\essinf_{x\in Q} u(x)
\end{equation}
for all cubes $Q\subset\Omega$ such that $4Q\subset\Omega$.
\end{theorem}

This setting motivates yet another simple application. Namely, using the local result (Theorem \ref{lokaali}), it is possible to obtain a single-weighted Poincar\'e inequality in cubes lying ``well inside'' $\Omega$ when the weight is a suitable supersolution to the $p$-Laplace equation. Observe that we can always choose $p=2$ below,
which leads to the classical Laplace equation.
\begin{theorem}\label{plapplication_thm}
Let $\Omega$ be a bounded domain, and  $\frac{2n}{n+1}<p<\infty$. Furthermore, let $w\in W^{1,p}_{\loc}(\Omega)$ be a  weak supersolution of the $p$-Laplace equation in $\Omega$ such that $w(x)>0$ for almost every  $x\in \Omega$, and $Q_0\subset\Omega$ a cube such that $4Q_0\subset\Omega$. The weighted Poincar\'e inequality
\begin{equation*}
\int_{Q_0}\abs{u(x)-u_{Q_0}}^pw(x)\dd x \leq Cl(Q_0)^p\int_{Q_0} \abs{\grad u(x)}^pw(x)\dd x
\end{equation*}
holds for every $u\in\Lip(Q_0)$ with $C = C(n, p) > 0$. 
\end{theorem}
\begin{proof} 
 We will check the assumptions of Theorem \ref{lokaali} for $v=\lvert Q_0\rvert^{\frac{p}{n}}w$, $p=q$, and $Q_0$ such that $4Q_0\subset\Omega$; such cubes will be referred to as \emph{admissible}. Whenever $Q_0$ is admissible, all dyadic subcubes $Q\in\mathcal{D}(Q_0)$ are naturally so as well. We remark that $v$ is also a weak supersolution to the $p$-Laplace equation in $\Omega$ such that $v(x)>0$ for almost every $x\in\Omega$. Write $\sigma = v^{-1/(p-1)}$. 

Fix an admissible cube $Q_0$ and $Q\in\mathcal{D}(Q_0)$.  Being a nonnegative supersolution, $w$ satisfies the inequality \eqref{halfnack} in $Q$.  In particular, letting $1\leq \beta=\beta(n,p)<\infty$ with $\beta(n-p)<n(p-1)$, we obtain a reverse H\"older inequality 
\begin{equation}\label{ainfty}
0 < \left(\dashint_{Q} w(x)^\beta\dd x\right)^\frac{1}{\beta} \leq C(n,p) \essinf_{x\in Q} w(x)\le C(n,p)\dashint_{Q} w(x) \dd x.
\end{equation}
Let $E\subset Q$ be a measurable set. By using first H\"older's inequality and then the reverse H\"older inequality \eqref{ainfty} with $\beta > 1$, which is possible since $p>\frac{2n}{n+1}$, we obtain
\begin{align*}
w(E)&\le \lvert E\rvert^{\frac{\beta-1}{\beta}}\biggl(\int_Q w(x)^{\beta}\dd x\biggr)^{\frac{1}{\beta}}\\
&\le C(n,p)\lvert E\rvert^{\frac{\beta-1}{\beta}}\lvert Q\rvert^{\frac{1}{\beta}}\dashint_Q w(x)\dd x\\
& = C(n,p)\biggl(\frac{\lvert E\rvert}{\lvert Q\rvert}\biggr)^{\frac{\beta-1}{\beta}}w(Q).
\end{align*}
Thus $w\in A^d_\infty(Q_0)$  with constants $\delta_w=\frac{\beta-1}{\beta}>0$ and $C_w=C(n,p)$. On the other hand, by H\"older's inequality, 
\[
1=\dashint_Q v(x)^{-\frac{1}{p}} v(x)^{\frac{1}{p}}\dd x\le \left(\dashint_Q v(x)^{-\frac{1}{p-1}}\dd x\right)^{\frac{p-1}{p}}\left(\dashint_Q v(x) \dd x\right)^{\frac{1}{p}}.
\]
Recall that $v(x)>0$ for almost every  $x\in \Omega$. Hence,  applying inequality \eqref{halfnack} with $u=v$ and $\beta=1$,  we obtain 
\begin{equation*}
\left(\dashint_Q v(x)^{-\frac{1}{p-1}} \dd x\right)^{-(p-1)}
\le \dashint_Q v(x)\dd x 
\le C(n,p)\essinf_{x\in Q} v(x).
\end{equation*}
Raising both sides to the negative power $-\frac{1}{p-1}$ yields
\begin{align*}
\esssup_{x\in Q}\sigma(x)&=\esssup_{x\in Q} v(x)^{-\frac{1}{p-1}}=
\left(\essinf_{x\in Q} v(x)\right)^{-\frac{1}{p-1}} \\&\le C(n,p)
\dashint_Q v(x)^{-\frac{1}{p-1}} \dd x
=C(n,p)
\dashint_Q \sigma(x) \dd x.
\end{align*}
As a consequence, we obtain for all measurable sets $E\subset Q$ that
\begin{align*}
\sigma(E)&\le \lvert E\rvert\esssup_{x\in Q} \sigma(x)\\
&\le C(n,p)\lvert E\rvert\left(\dashint_Q \sigma(x)\dd x\right)  =C(n,p)\frac{\lvert E\rvert}{\lvert Q\rvert}\sigma(Q).
\end{align*}
Thus $\sigma\in A^d_\infty(Q_0)$ with constants $\delta_\sigma=1$ and $C_\sigma=C(n,p)$. 

The condition \eqref{koo0} is verified next. Fix a cube $Q\in\mathcal{D}(Q_0)$, recall that $-\frac{1}{p-1}< 0$ and $v=\lvert Q_0\rvert^{\frac{p}{n}}w$, and apply $\eqref{ainfty}$:
\begin{align*}
& \left(\frac{1}{\abs{Q}^{1-1/n}}\right)^p\cdot \int_Q w(x)\dd x \cdot \left(\int_Q v(x)^{-\frac{1}{p-1}}\dd x\right)^{p-1}\\
& \hspace*{3em} \leq C(n,p)\left(\frac{1}{\abs{Q}^{1-1/n}}\right)^p\cdot \abs{Q}\essinf_{x\in Q}w(x)\cdot \abs{Q}^{p-1} \lvert Q_0\rvert^{-\frac{p}{n}}\left(\essinf_{x\in Q} w(x)\right)^{-1}\\
& \hspace*{3em} = C(n,p)\abs{Q}^\frac{p}{n}\abs{Q_0}^{-\frac{p}{n}} \leq C(n,p).
\end{align*}
The result then follows from Theorem \ref{lokaali}.
\end{proof}


\end{document}